\documentclass{article}

\usepackage{amsfonts}

\usepackage{amsthm}

\usepackage{amssymb}

\usepackage{amsmath}

\usepackage{enumerate}

\usepackage[all]{xy}

\usepackage{graphicx}

\newcommand{\N}{\mathbb{N}}

\newcommand{\Z}{\mathbb{Z}}

\newcommand{\R}{\mathbb{R}}

\newcommand{\C}{\mathbb{C}}

\newcommand{\Ga}{\Gamma}

\newcommand{\gi}{\text{girth}}

\newcommand{\wgn}{\widetilde{G_n}}

\newcommand{\Hi}{\mathcal{H}_0}

\theoremstyle{definition} \newtheorem{sg}{Definition}[section]

\theoremstyle{definition} \newtheorem{exdef}[sg]{Definition}

\theoremstyle{definition} \newtheorem{mex}[sg]{Definition / Theorem}

\theoremstyle{definition} \newtheorem{gidef}[sg]{Definition}

\theoremstyle{plain} \newtheorem{girthmain}[sg]{Theorem}

\theoremstyle{plain} \newtheorem{unithe}[sg]{Theorem}

\theoremstyle{plain} \newtheorem{gmcor}[sg]{Corollary}

\theoremstyle{definition} \newtheorem{cc}{Definition}[section]

\theoremstyle{definition} 

\theoremstyle{definition} 

\theoremstyle{definition} \newtheorem{cf}[cc]{Definition}

\theoremstyle{remark} \newtheorem{cfex}[cc]{Examples}

\theoremstyle{definition} \newtheorem{udbg}{Definition}[section]

\theoremstyle{definition} \newtheorem{roe}[udbg]{Definition}

\theoremstyle{remark} \newtheorem{rar}[udbg]{Remark}

\theoremstyle{definition} \newtheorem{ce}[udbg]{Definition}

\theoremstyle{plain} \newtheorem{celem}[udbg]{Lemma}

\theoremstyle{definition} \newtheorem{equiroe}[udbg]{Definition}

\theoremstyle{plain} \newtheorem{morlem}[udbg]{Lemma}

\theoremstyle{plain} \newtheorem{philem}[udbg]{Lemma}

\theoremstyle{plain} \newtheorem{maxcor}[udbg]{Corollary}

\theoremstyle{definition} \newtheorem{onl}[udbg]{Definition}

\theoremstyle{plain} \newtheorem{treelem}[udbg]{Lemma}

\theoremstyle{plain} \newtheorem{onllem}[udbg]{Lemma}

\theoremstyle{definition} \newtheorem{ghostdef}{Definition}[section]

\theoremstyle{remark} \newtheorem{gbgrem}[ghostdef]{Remark}

\theoremstyle{remark} \newtheorem{ghostex}[ghostdef]{Examples}

\theoremstyle{remark} \newtheorem{ghostprob}[ghostdef]{Problem}

\theoremstyle{plain} \newtheorem{ghostlem}[ghostdef]{Lemma}

\theoremstyle{plain} \newtheorem{kazdies}[ghostdef]{Lemma}

\theoremstyle{definition} \newtheorem{k0hom}{Definition}[section]

\theoremstyle{definition} \newtheorem{k0a}[k0hom]{Definition}

\theoremstyle{definition} \newtheorem{rips}[k0hom]{Definition}

\theoremstyle{plain} \newtheorem{indlem}[k0hom]{Lemma}

\theoremstyle{plain} \newtheorem{bclem}[k0hom]{Lemma}

\theoremstyle{plain} \newtheorem{negthe}{Theorem}[section]

\theoremstyle{plain} \newtheorem{negcor}[negthe]{Corollary}

\theoremstyle{remark} \newtheorem{taupprob}[negthe]{Problem}

\theoremstyle{plain} \newtheorem{atiyahthe}[negthe]{Theorem}

\theoremstyle{remark} 

\theoremstyle{plain} \newtheorem{atlem}[negthe]{Lemma}

\theoremstyle{plain} \newtheorem{nonbg}[negthe]{Proposition}

\theoremstyle{plain} \newtheorem{posthe}{Theorem}[section]

\theoremstyle{remark} \newtheorem{injgen}[posthe]{Remark}

\theoremstyle{plain} \newtheorem{asinj}[posthe]{Lemma}

\theoremstyle{plain} \newtheorem{hkthe}[posthe]{Proposition}

\theoremstyle{definition} \newtheorem{gm}{Definition}[section] 

\theoremstyle{plain} \newtheorem{bccoeff}[gm]{Theorem}

\theoremstyle{definition} \newtheorem{adef}[gm]{Definition}

\theoremstyle{plain} \newtheorem{alem}[gm]{Lemma}

\theoremstyle{plain} \newtheorem{alem2}[gm]{Lemma}

\theoremstyle{definition} \newtheorem{unialg}{Definition}[section]

\theoremstyle{plain} \newtheorem{unimor}[unialg]{Proposition}

\theoremstyle{plain} \newtheorem{unitheo}[unialg]{Theorem}

\theoremstyle{plain} \newtheorem{unicor}[unialg]{Corollary}

\title{Higher index theory for certain expanders and Gromov monster groups I}

\author{Rufus Willett and Guoliang Yu\footnote{Partially supported by the NSF.}}

\date{}

\begin{document}

\maketitle

\begin{abstract}
In this paper, the first of a series of two, we continue the study of higher index theory for expanders.  We prove that if a sequence of graphs is an expander and the girth of the graphs tends to infinity, then the coarse Baum-Connes assembly map is injective, but not surjective, for the associated metric space $X$.  

Expanders with this girth property are a necessary ingredient in the construction of the so-called `Gromov monster' groups that (coarsely) contain expanders in their Cayley graphs. We use this connection to show that the Baum-Connes assembly map with certain coefficients is injective but not surjective for these groups.  Using the results of the second paper in this series, we also show that the maximal Baum-Cones assembly map with these coefficients is an isomorphism.
\end{abstract}

\tableofcontents

\section{Introduction}

The \emph{coarse Baum-Connes conjecture} postulates an algorithm for computing the higher indices of generalized elliptic operators on non-compact spaces.  More precisely, it claims that a certain coarse assembly map
$$
\lim_{R\to\infty}K_*(P_R(X))\to K_*(C^*(X))
$$
for a metric space $X$ is an isomorphism \cite{Roe:1993lq,Roe:1996dn}; this depends only on the large scale (or \emph{coarse}) geometry of $X$.  The right hand side above is a `noncommutative object' (the $K$-theory of a certain $C^*$-algebra, the \emph{Roe algebra of $X$}) and the left hand side is a `commutative object' (a limit of the $K$-homology groups of certain spaces, the \emph{Rips complexes of $X$}); seen in this way, the coarse Baum-Connes conjecture forms a bridge between Connes's theory of noncommutative geometry  \cite{Connes:1994zh} and classical topology and geometry. The conjecture has many applications, including to the Novikov conjecture when $X$ is a finitely generated group $\Ga$ equipped with a word metric, and to the existence of positive scalar curvature metrics when $X$ is a Riemannian manifold.  The \emph{coarse Novikov conjecture}, which states that the coarse assembly map is an injection, is sufficient for many of these applications, for example to positive scalar curvature.

In this paper, we will study spaces $X$ as in the definition below.
\begin{sg}\label{sg}
A metric space $X$ is called a \emph{space of graphs} if it is an infinite disjoint union 
$$
X=\sqcup_{n\in\N}G_n,
$$
where each $G_n$ is a finite connected graph and $X$ is equipped with any metric $d$ such that 
\begin{itemize}
\item the restriction of $d$ to each $G_n$ is the edge metric;
\item the pairwise distances $d(G_n,G_m)$ tend to infinity as $n+m\to\infty$ through pairs with $n\neq m$
\end{itemize}
(any metric satisfying these two conditions will give rise to the same coarse geometric structure on $X$).  
\end{sg}
\noindent
For our precise conventions on graphs and edge metrics, see Section \ref{cgsec} below.

Using cutting-and-pasting arguments (see e.g.\ \cite{Higson:1993th} and \cite{Yu:1998wj}) and the fact that any `reasonable' metric space is equivalent to a graph in coarse geometry, a lot of information on the general coarse Baum-Connes conjecture, and on many other coarse geometric properties, can be deduced from such spaces of graphs $X$.

It is a well-known fact that spaces of graphs are generically \emph{expanders} in the sense of the following definition -- see for example \cite{Lubotzky:1994tw}.

\begin{exdef}\label{exdef}
Let $(G_n)_{n\in\N}$ be a sequence of finite graphs.  The \emph{graph Laplacian}, denoted $\Delta_n$, is the bounded operator on $l^2(G_n)$ defined by the formula
$$
(\Delta_n f)(x)=f(x)-\sum_{d(x,y)=1}\frac{f(y)}{\sqrt{\text{degree}(x)\text{degree}(y)}};
$$
$\Delta_n$ is a positive operator of norm at most $2$ (see for example \cite[Section 5.2]{Bekka:2000kx} -- our formula is an adaptation of the one there to Hilbert spaces built from a uniform counting measure).
The space of graphs $X=\sqcup G_n$, or the sequence $(G_n)$, is called an \emph{expander} if the following hold:
\begin{enumerate}[(i)]
\item there exists $k\in\N$ so that all the vertices in each $G_n$ have degree at most $k$;
\item the cardinalities $|G_n|$ tend to infinity as $n$ tends to infinity;
\item there exists $c>0$ such that $\text{spectrum}(\Delta_n)\subseteq\{0\}\cup[c,1]$ for all $n$.
\end{enumerate}
The space $X$, or the sequence $(G_n)$, is called a \emph{weak expander} if conditions (ii) and (iii), but not necessarily (i), hold.
\end{exdef}
\noindent
Expanders have many applications in information theory and both applied and theoretical computer science.

Note that although expanders are generic, it is difficult to explicitly construct them.  This was first achieved by Margulis \cite{Margulis:1973lh}, using discrete groups with (relative) property (T). A variant of Margulis's construction proceeds as follows.  

\begin{mex}\label{mex}
Let $\Gamma$ be an infinite finitely generated group and $(\Gamma_n)_{n\in\N}$ be a sequence of finite index normal subgroups such that $\Ga$ has \emph{property $(\tau)$} (\cite[Chapter 4]{Lubotzky:1994tw}) with respect to $(\Ga_n)$ and so that $|\Gamma/\Gamma_n|\to\infty$ as $n\to\infty$.  Equip each of the finite groups $\Ga/\Ga_n$ with a graph structure by considering its Cayley graph with respect to the image of some fixed finite generating set of $\Ga$. Then the space of graphs
\begin{equation}\label{fromgroup}
X=\sqcup_{n\in\N}\Gamma/\Gamma_n
\end{equation} 
is an expander.  We call expanders arising in this way \emph{Margulis-type expanders}.
\end{mex}

As we have been discussing so far, spaces of graphs are of interest in of themselves for the coarse Baum-Connes conjecture.  Thus the fact that expanders are generic amongst spaces of graphs motivates their study in this area.  Another major motivation is that  expanders have interesting pathological properties with regard to the $K$-theory of the associated Roe $C^*$-algebras.  As a result of these two points, there has been much work on expanders from the point of view of the Baum-Connes and coarse Baum-Connes conjectures.  
\begin{itemize}
\item On the negative side, Higson \cite{Higson:1999km} showed that the coarse assembly map was not surjective for certain Margulis-type expanders.  In \cite[Section 6]{Higson:2002la}, Higson--Lafforgue--Skandalis used groupoid techniques to show that for any expander $X$, either the coarse Baum-Connes assembly map for $X$ fails to be surjective, or the Baum-Connes assembly map with certain coefficients for an associated groupoid fails to be injective; they also show that in the case of certain Margulis-type expanders, the former always occurs.  Also in the negative direction, \v{S}pakula has exhibited further pathological properties by showing that uniform Roe algebras associated to certain expanders are not $K$-exact \cite{Spakula:2009rr}.  Note that a counterexample to the coarse Novikov conjecture for \emph{non}-bounded-geometry spaces (not obviously related to expanders) was given by the second author in \cite[Section 8]{Yu:1998wj}.
\item On the positive side, injectivity of the coarse assembly map (i.e.\ the coarse Novikov conjecture), or of its maximal version, is known to be true for certain classes of expanders by work of Gong--Wang--Yu \cite{Gong:2008ja}, Chen--Tessera--Wang--Yu \cite{Chen:2008so}, Guentner--Tessera--Yu \cite{Guentner:2008gd} and Oyono-Oyono--Yu \cite{Oyono-Oyono:2009ua}.  Moreover, the work of Oyono-Oyono--Yu cited above also proves isomorphism of the maximal version of the coarse assembly map.
\end{itemize}
\noindent
Most\footnote{The exception is the paper of Higson--Lafforgue--Skandalis -- see \cite[Proposition 10]{Higson:2002la} -- but the sharpest version of their result requires Margulis-type expanders -- cf.  \cite[Proposition 11]{Higson:2002la} and preceeding discussion.} of these results only apply to the Margulis-type expanders from Definition \ref{mex} above.  It is desirable to prove analogues of these results where the only assumptions on $X$ are graph theoretic, or coarse geometric.

This is partly achieved in this paper.  Recall first that the \emph{girth} of a graph $G$, denoted $\text{girth}(G)$, is the shortest length of a cycle in $G$.
\begin{gidef}\label{gidef}
Let $X=\sqcup G_n$ be a space of graphs as in Definition \ref{sg}.  The space $X$ (or the sequence $(G_n)_{n\in\N}$) is said to have \emph{large girth} if $\text{girth}(G_n)\to\infty$ as $n\to\infty$.
\end{gidef}

Our main aim is the following very natural result; for precise statements, see Theorems \ref{negthe} and \ref{posthe} below, and Theorem 1.1 from the second paper in this series \cite{Willett:2010zh}.

\begin{girthmain}\label{girthmain}
Let $X=\sqcup G_n$ be a space of graphs as in Definition \ref{sg} above, with large girth as in Definition \ref{gidef}.
\begin{enumerate}[(i)]
\item If $X$ is a weak expander, then the coarse assembly map fails to be surjective for $X$.
\item The coarse assembly map is injective for $X$ (i.e.\ the \emph{coarse Novikov conjecture} holds for $X$).
\item If there is a uniform bound on the vertex degrees of the graphs $G_n$, then the maximal coarse assembly map is an isomorphism for $X$.
\end{enumerate}
\end{girthmain}  
\noindent
Most of our results apply somewhat more generally than to graphs with large girth: see  Theorem \ref{negthe} and Remark \ref{injgen} below, and also Remark 3.1 from the second paper in this series \cite{Willett:2010zh}.  

\v{S}pakula \cite{Spakula:2009tg} has developed a version of the coarse assembly map for the uniform Roe algebra.  We will not prove the following theorem in full detail, but Appendix \ref{unicase} collects together the necessary adjustments to the proof of Theorem \ref{girthmain}.

\begin{unithe}\label{unithe}
Let $X=\sqcup G_n$ be a space of graphs as in Definition \ref{sg} above, with large girth as in Definition \ref{gidef}.
\begin{enumerate}[(i)]
\item If $X$ is a weak expander, then the uniform coarse assembly map fails to be surjective for $X$.
\item The uniform coarse assembly map is injective for $X$.
\item If there is a uniform bound on the vertex degrees of the graphs $G_n$, then the maximal uniform coarse assembly map is an isomorphism for $X$.
\end{enumerate}
\end{unithe}  

There are many explicit examples of expanders with girth tending to infinity coming from property $(\tau)$ groups, and related constructions: a particularly nice one is the sequence of \emph{Ramanujan graphs} found by Lubotzky, Phillips and Sarnak \cite{Lubotzky:1988fu}.  

On the other hand, in the second paper in this series we introduce a new property, called \emph{geometric property (T)}, which is an obstruction to isomorphism of the maximal coarse assembly map.  Geometric property (T) is thus in some sense a strong opposite to the property of having large girth.  Using geometric property (T), we deduce that no Margulis-type expander coming from a property (T) (as opposed to property $(\tau)$) group can be coarsely equivalent to an expander with large girth; in particular, there are a large class of expanders to which the methods of the current work cannot apply.  See Corollary 7.4 in the second paper in this series \cite{Willett:2010zh} for this and some other purely geometric corollaries.

Also important to us is the fact that expanders with this girth property are the central ingredient in Gromov's construction of groups that do not coarsely embed in Hilbert space \cite{Gromov:2003gf}.  A complete exposition of Gromov's construction has recently been provided by Arzhantseva and Delzant \cite{Arzhantseva:2008bv}.  Using the relationship between the Baum-Connes and coarse Baum-Connes conjectures \cite{Yu:1995zl}, the theorem above has the following corollary (see Section \ref{bcsec} below for details).

\begin{gmcor}\label{gmcor}
Say $\Ga$ is a countable discrete group containing a coarsely embedded sequence of expanders with large girth (in particular, any of the groups shown to exist using the methods of Gromov).  Then there exists a $\Ga$-$C^*$-algebra $A$ such that:
\begin{enumerate}[(i)]
\item the Baum-Connes assembly map for $\Ga$ with coefficients in $A$ is injective;
\item the Baum-Connes assembly map for $\Ga$ with coefficients in $A$ is not surjective;
\item the maximal Baum-Connes assembly map for $\Ga$ with coefficients in $A$ is an isomorphism.
\end{enumerate}
Moreover, using Theorem \ref{unithe}, one may take the coefficient algebra $A$ to be commutative.
\end{gmcor}
\noindent
The contrast with the assembly map without coefficients for $\Ga$ is striking: here little is known about the usual (reduced) assembly map, while the maximal assembly map is injective, but not surjective (assuming, as we may, that $\Ga$ has property (T)).  The existence of an example for which the usual Baum-Connes conjecture (with coefficients) fails, but its maximal version is true, is also suggestive of new phenomena in noncommutative harmonic analysis.

The introduction of geometric property (T) and our study of ghost operators lead to several new open questions, which we state as open problems at various points in the main piece -- see \ref{ghostprob} and \ref{taupprob} below and 7.5, 7.6 and 7.7 in the second paper \cite{Willett:2010zh}.  We have made a deliberate effort to make the piece as accessible as possible, while trying to keep its length under some sort of control.

\subsection*{Outline of the piece}

This is the first of a series of two papers.  It deals with the necessary background, and injectivity and surjectivity results for the coarse assembly map for expanders with large girth.  It also connects these results to the Baum-Connes conjecture with coefficients for Gromov monster groups.  The second paper in the series proves that the maximal coarse assembly map is an isomorphism for spaces of graphs with large girth.  Combined, these two papers give a fairly complete picture of the higher indices coming from this particular class of expanders.

Sections \ref{cgsec} to \ref{asssec} mainly cover background material.
Section \ref{cgsec} gives our conventions on graphs and covering spaces, and introduces the property of \emph{asymptotic faithfulness} for a sequence of covers; this underpins much of the rest of the paper.  Section \ref{algsec} introduces the main versions of the \emph{Roe algebra} of a metric space, a $C^*$-algebra that captures the coarse geometry of the space and whose $K$-theory is a receptacle for higher indices, that we will use throughout the piece.  It also introduces the \emph{operator norm localization property} of Chen--Tessera--Wang--Yu \cite{Chen:2008so} that will be another important tool.  Section \ref{asssec} gives the basic background on assembly maps that we will need later in the piece.

Sections \ref{ghostsec} to \ref{injsec} discuss surjectivity and injectivity of the coarse assembly map.  Section \ref{ghostsec} introduces \emph{ghost operators}, a class of highly `non-local' operators, and constructs non-trivial (i.e.\ non-compact) examples in the Roe algebras of certain spaces.  Section \ref{surjsec} shows that, under certain hypotheses guaranteeing asymptotic faithfulness and the operator norm localization property, $K$-theory classes coming from non-compact ghost operators cannot appear in the image of the coarse assembly map; in the cases that such operators exist (for example, in the case of expanders), this yields counterexamples to surjectivity of the coarse Baum-Connes conjecture.  Our analysis in this section is based on ideas of Higson \cite{Higson:1999km}.  Section \ref{injsec} proves injectivity of the coarse assembly map for sequences of graphs with large girth; the essential ingredients are asymptotic faithfulness, the operator norm localization property, and the strong Novikov conjecture for free groups.

Section \ref{bcsec} uses an identification of the Baum-Connes assembly map for a \emph{Gromov monster} group with certain coefficients, and the coarse Baum-Connes assembly map for an expander with large girth, to apply our results to the Baum-Connes conjecture for such groups.  Finally, Appendix \ref{unicase} discusses the necessary changes to extend our results to the uniform case studied by \v{S}pakula \cite{Spakula:2009tg}; in particular, this implies that our results on the Baum-Connes conjecture can be made to work with commutative coefficients.

\section{Covers and graphs}\label{cgsec}

In this section we set up basic terminology about graphs and coverings of graphs.  Some of this is slightly non-standard: for example, we identify a graph with its vertex set, equipped with some additional structure.

If $(X,d)$ is a metric space, $x\in X$ and $R>0$ we denote by
$$
B(x,R)=\{y\in X~|~d(x,y)<R\}
$$
the open ball about $x$ of radius $R$.

The following definition is central to this piece.

\begin{cc}\label{cc}
Let $X$ be a metric space, and let $\pi:\widetilde{X}\to X$ be a surjective map.  Let $R>0$.  Then $(\widetilde{X},\pi)$ is called an \emph{R-metric cover} of $X$ if for all $x\in \widetilde{X}$, the restriction of $\pi$ to the ball $B(x,R)$ of radius $R$ around $x$ in $\widetilde X$ is an isometry onto $B(\pi(x),R)$, the ball of radius $R$ about $\pi(x)$ in $X$. 
\end{cc}

We now specialize to graphs.  For us, a \emph{graph} $G$ will consist as a set of the collection of zero-simplices in an unoriented one-dimensional simplicial complex such that every \emph{vertex} (i.e.\ zero-simplex) is a face of only finitely many \emph{edges} (i.e.\ one-simplices). The number of edges each vertex is a face of is called its \emph{degree}.  We write $x,y\in G$ for vertices, and $(x,y)$ for the (necessarily unique) edge connecting $x$ and $y$.  The \emph{edge metric} on a graph $G$ (recalling that $G$ simply denotes the vertex set) is defined by
\begin{align*}
d(x,y)=\min\{n~|~\text{there exist } &  x=x_0,...,x_n= y\\ & \text{ such that each $(x_i,x_{i+1})$ is an edge}\}.
\end{align*}

If we ever want to discuss the graph as a one-dimensional topological simplicial complex, we refer to the \emph{simplicial graph}.   This is useful to make sense of notions such as \emph{covering space}, \emph{Galois covering space}, \emph{universal cover}, \emph{covering group} etc.\ from algebraic topology.  For example, by a \emph{Galois covering space} of $G$ we mean a graph $\widetilde{G}$ such that the simplicial graph associated to $\widetilde{G}$ is a Galois covering space of the simplicial graph associated to $G$ via a covering map that is a simplicial map; this implies of course that (the vertex set) $\widetilde{G}$ is equipped with a map to $G$, and an action of a covering group by deck transformations.

\begin{cf}\label{cf}
Let $X=\sqcup G_n$ be a space of graphs.  We call a sequence $\widetilde{X}=(\wgn)$ a \emph{covering sequence for $X$} if $\wgn$ is a Galois covering graph of $G_n$ for each $n$.  Denote by $\pi_n:\wgn\to G_n$ the associated covering maps.  

The sequence $\widetilde{X}$ is said to be \emph{asymptotically faithful} if for all $R>0$ there exists $N_R\in\N$ such that for all $n\geq N_R$, the map $\pi_n:\widetilde{G_n}\to G_n$ is an $R$-metric cover.
\end{cf}

The following examples are important for us.

\begin{cfex}\label{cfex}
\begin{enumerate}[(i)]
\item Let $G$ be a graph.  A \emph{cycle of length $n$} in $G$ is a finite ordered set $$\{(x_0,x_1),(x_1,x_2),...,(x_{n-1},x_n)\}$$ of edges of $G$ such that $x_0=x_n$ but $x_i\neq x_j$ for all $0<i<j<n$.  The \emph{girth} of $G$ is the shortest length of a cycle in $G$.  Note that if the girth of $G$ is $g$, then the covering map $\pi:\widetilde{G}\to G$   from the universal cover $\widetilde{G}$ of $G$ (a tree) to $G$ itself is a $\lfloor g/2\rfloor$-metric cover.  

In particular, if $(G_n)_{n\in\N}$ is a sequence of graphs such that girth$(G_n)\to\infty$ as $n\to\infty$ and $(\wgn)_{n\in\N}$ is the corresponding sequence of universal covers, then $\widetilde{X}=(\wgn)_{n\in\N}$ is an asymptotically faithful covering sequence for $X=\sqcup G_n$; indeed, the covering sequence given by the universal covers is asymptotically faithful if and only if $\gi(G_n)\to\infty$.
\item Say $\Gamma$ is a discrete group, equipped with a finite generating set $S$, and identified with the (vertex set of the) Cayley graph associated to this generating set (the edge metric then agrees with the word metric on $\Gamma$ coming from $S$).  Let $(\Ga_n)_{n\in\N}$ be a sequence of finite index normal subgroups such that for all $R>0$ there exists $N_R\in\N$ so that if $B(e,R)$ is the ball in $\Ga$ of radius $R$ about the identity, then $\Ga_n\cap B(e,R)=\{e\}$ for all $n\geq N_R$.  Let $G_n$ be the Cayley graph of  $\Gamma/\Gamma_n$ taken with respect to the image of the generating set $S$ (assume $(S\cup S^2)\cap \Gamma_n=\{e\}$ for all $n$ to avoid pathologies).  Then the constant sequence $\widetilde{X}=(\Gamma)_{n\in\N}$ is an asymptotically faithful covering sequence for $X=\sqcup G_n$.
\end{enumerate}
\end{cfex}

\section{Roe algebras and the operator norm localization property}\label{algsec}

In this section we introduce some versions of the \emph{Roe algebra} of a metric space $X$, a $C^*$-algebra originally defined by Roe \cite{Roe:1993lq,Roe:1996dn}.  We also discuss the \emph{operator norm localization property} of Chen, Tessera, Wang and the second author \cite{Chen:2008so}, which will be useful to relate properties of a space of graphs to those of sequences of covers.

\begin{udbg}\label{udbg}
Let $(X,d)$ be a metric space.  $X$ is said to be \emph{$\delta$-separated} for some $\delta>0$ if $d(x,y)\geq\delta$ for all $x,y\in X$ with $x\neq y$; $X$ is said to be \emph{uniformly discrete} if it is $\delta$-separated for some $\delta>0$.  If $X$ is uniformly discrete, it is said to be of \emph{bounded geometry} if for all $R>0$ there exists $N_R\in\N$ such that $|B(x,R)|\leq N_R$.  

A general proper metric space $X$ has \emph{bounded geometry} if for some $\delta>0$ some (equivalently, any) maximal $\delta$-separated subspace of $X$ has bounded geometry in the sense above. 
\end{udbg}

Fix for the rest of the piece $\Hi$, an infinite dimensional separable Hilbert space, and let $\mathcal{K}:=\mathcal{K}(\Hi)$ be the algebra of compact operators on $\Hi$.  The following algebras are important partly as their $K$-theory provides a receptacle for higher indices of elliptic operators.  The definition is due to Roe \cite{Roe:1993lq,Roe:1996dn}.

\begin{roe}\label{roe}
Let $X$ be a proper metric space, and fix $Z$ a countable dense subset of $X$.  Let $T$ be a bounded operator on $l^2(Z,\Hi)$ and write $T=(T_{x,y})_{x,y\in Z}$ so that each $T_{x,y}$ is an element of $\mathcal{B}(\Hi)$.   $T$ is said to be \emph{locally compact} if:
\begin{enumerate}
\item all the matrix entries $T_{x,y}$ are in $\mathcal{K}(\Hi)$;
\item for any bounded subset $B\subseteq X$, the set
$$
\{(x,y)\in (B\times B)\cap (Z\times Z)~|~T_{x,y}\neq 0\}
$$
is finite;
\end{enumerate}
The \emph{propagation of $T$} is defined to be
$$
\text{prop}(T):=\inf\{S>0~|~T_{x,y}=0 \text{ for all } x,y\in Z \text{ with } d(x,y)>S\}.
$$
The \emph{algebraic Roe algebra of $X$}, denoted $\C[X]$, is the $*$-subalgebra of $\mathcal{B}(l^2(Z,\Hi))$ consisting of all locally compact operators of finite propagation.  The \emph{Roe algebra of $X$}, denoted $C^*(X)$, is the closure of $\C[X]$ inside $\mathcal{B}(l^2(Z,\Hi))$.

Say now that $X$ has bounded geometry.  The \emph{maximal Roe algebra}, denoted $C^*_{max}(X)$, is the completion of $\C[X]$ for the norm
$$
\|T\|_{max}:=\sup\{\|\pi(T)\|_{\mathcal{B}(\mathcal{H})}~|~\pi:\C[X]\to\mathcal{B}(\mathcal{H}) \text{ a $*$-representation}\}
$$
(the bounded geometry assumption is sufficient for this expression to be finite -- see for example \cite[Lemma 3.4]{Gong:2008ja}).
\end{roe}

\begin{rar}\label{rar}
If $X$ is uniformly discrete, as will be the case in most of our examples, we have no choice but to take $Z=X$ in the above.  This simplifies the definition in this case.  The definition above is used, however, as we will sometimes need variants of the Roe algebra  that take \emph{both} the local and large-scale structure of $X$ into account, and it allows for a uniform treatment.
\end{rar}

The following definition introduces the natural notions of `injection' and `isomorphism' in coarse geometry.

\begin{ce}\label{ce}
Let $(X,d_X)$ and $(Y,d_Y)$ be metric spaces.  A (not necessarily continuous) map $f:X\to Y$ is said to be a \emph{coarse embedding} if there exist non-decreasing functions $\rho_{\pm}:\R_+\to \R_+$ such that $\rho_\pm(t)\to\infty$ as $t\to\infty$ and
$$
\rho_-(d_X(x,y))\leq d_Y(f(x),f(y))\leq \rho_+(d_X(x,y))
$$
for all $x,y\in X$.  The spaces $X$ and $Y$ are said to be \emph{coarsely equivalent} if there exist coarse embeddings $f:X\to Y$ and $g:Y\to X$ and a constant $C\geq 0$ such that 
$$
d_X(x,g(f(x)))\leq C,~~d_Y(y,f(g(y)))\leq C
$$
for all $x\in X$ and $y\in Y$.
\end{ce}

See for example \cite[Secton 4, Lemma 3]{Higson:1993th} for the $K$-theory part of the following lemma, which is all we will use (the algebraic part is in any case not difficult).

\begin{celem}\label{celem}
Up to non-canonical isomorphism, $\C[X]$, $C^*(X)$ and $C^*_{max}(X)$ do not depend on the choice of dense subspace $Z$, and moreover only depend on $X$ itself up to coarse equivalence.  Up to canonical isomorphism, the $K$-theory groups $K_*(C^*(X))$ and $K_*(C^*_{max}(X))$ do not depend on the choice of $Z$, and moreover only depend on $X$ itself up to coarse equivalence.  \qed
\end{celem}

In the presence of a discrete group action, the $K$-theory groups of the following algebras are receptacles for equivariant higher indices of elliptic operators. 

\begin{equiroe}\label{equiroe}
Let $X$ be a proper metric space, and $\Gamma$ a countable discrete group acting freely and properly on $X$ by isometries.  Fix a $\Gamma$-invariant countable dense subset $Z\subseteq X$, and use this to define $\C[X]$ as in Definition \ref{roe} above. The \emph{equivariant algebraic Roe algebra of $X$}, denoted $\C[X]^\Gamma$ is the $*$-subalgebra of $\C[X]$ consisting of $\Gamma$ invariant matrices $(T_{x,y})$, i.e.\ those that satisfy $T_{gx,gy}=T_{x,y}$ for all $g\in\Gamma$ and $x,y\in Z$.  $\C[X]^\Ga$ does not depend on the choice of $Z$ up to non-canonical isomorphism.

The \emph{equivariant Roe algebra of $X$}, denoted $C^*(X)^\Gamma$, is the completion of $\C[X]^\Gamma$ for its natural representation on $l^2(Z,\Hi)$.

Say now that $X$ has bounded geometry.  The \emph{maximal equivariant Roe algebra of $X$}, denoted $C^*_{max}(X)^\Gamma$, is the completion of $\C[X]^\Gamma$ for the norm 
$$
\|T\|_{max}=\sup\{\|\pi(T)\|_{\mathcal{B}(\mathcal{H})}~|~\pi:\C[X]^\Gamma\to\mathcal{B}(\mathcal{H}) \text{ a $*$-representation}\}.
$$
\end{equiroe}
\noindent
Note that, despite the notation, $C^*(X)^\Gamma$ is not defined to consist of the $\Ga$-fixed points in $C^*(X)$: indeed, it can happen that $C^*_{max}(X)^\Ga$ is \emph{not} equal to the $\Ga$-fixed points in $C^*_{max}(X)$; it is suspected that this sort of phenomenon can also occur for $C^*(X)^\Ga$, but no examples are known.  The assumption that the $\Ga$ action on $X$ is free is not really necessary, but in this case the `correct' definition of $\C[X]^\Ga$ is a little more complicated, and the free case is all we need.

Recall now that if $\Ga$ is a discrete group, then its \emph{group algebra} $\C[\Ga]$ is the $*$-algebra of all finite formal linear combinations $\sum_{g\in\Ga}\lambda_gu_g$, where $\lambda_g\in \C$ and the $u_g$ satisfy $u_gu_h=u_{gh}$ and $u_g^*=u_{g^{-1}}$.  The \emph{reduced group $C^*$-algebra}, $C^*_r(\Ga)$, is the completion of $\C[\Ga]$ for its natural representation on $l^2(\Ga)$ by left shifts, and the \emph{maximal group $C^*$-algebra}, $C^*_{max}(\Ga)$, is the completion of $\C[\Ga]$ for the norm coming from the supremum over all $*$-representations. 
The equivariant Roe algebra is related to the reduced group $C^*$-algebra by the following well-known lemma; see for example \cite[Lemma 5.14]{Roe:1996dn}. 

\begin{morlem}\label{morlem}
Say $\Gamma$ acts properly, freely and cocompactly by isometries on a proper metric space $X$.  Let $Z\subseteq X$ be the countable, dense, $\Ga$-invariant subset used to define $\C[X]^\Ga$.  Let $D\subseteq Z$ be a precompact fundamental domain for the $\Ga$-action, by which we mean that each $\Ga$ orbit contains precisely one element of $D$, and that the closure of $D$ in $X$ is compact.  

Then there is an $*$-isomorphism 
$$
\psi_D:C^*(X)^\Ga\to C_r^*(\Ga)\otimes\mathcal{K}(l^2(D,\Hi)).
$$
Moreover, the isomorphism on $K$-theory induced by this $*$-isomorphism is independent of the choice of $D$.
\end{morlem}

As is again well-known, in the situation of the lemma there is actually a canonical Morita equivalence between $C^*(X)^\Ga$ and $C^*_r(\Ga)$: see \cite[Lemma 2.3]{Roe:2002nx}.  The isomorphism above seems more useful for computations, however.

\begin{proof}
Let $\mathcal{K}_f(l^2(D,\Hi))$ be the dense $*$-subalgebra of $\mathcal{K}(l^2(D,\Hi))$ such that if $K\in \mathcal{K}_f(l^2(D,\Hi))$ is written as a matrix $(K_{x,y})_{x,y\in D}$, then only finitely many entries are non-zero.  Let $T$ be an element of $\C[X]^\Ga$, and for each $g\in\Ga$ define an element $T^{(g)}$ of $\mathcal{K}_f(l^2(D,\Hi))$ by the matrix formula
$$
T^{(g)}_{x,y}:=T_{x,gy} ~~\text{ for all $x,y\in D$}.
$$
Define now a $*$-homomorphism
$$
\psi_D:\C[X]^\Ga\to \C[\Ga]\odot \mathcal{K}_f(l^2(D,\Hi)),
$$
where the right hand side is the algebraic tensor product of the group algebra $\C[\Ga]$ and $\mathcal{K}_f(l^2(D,\Hi))$, by the formula
$$
T\mapsto \sum_{g\in\Ga}u_g\odot T^{(g)};
$$
using finite propagation of $T$, only finitely many of the operators $T^{(g)}$ are non-zero, so this makes sense.  It is not hard to check that $\psi_D$ is in fact a $*$-isomorphism.  Moreover, if $\psi_D$ is used to identify these two $*$-algebras then with respect to the isomorphism $l^2(Z,\Hi)\cong l^2(\Ga)\otimes l^2(D,\Hi)$ the representation of $\C[X]^\Ga$ on $l^2(Z,\Hi)$ corresponds to the natural one of $\C[\Ga]\odot \mathcal{K}_f(l^2(D,\Hi))$ on $l^2(\Ga)\otimes l^2(D,\Hi)$.  This shows that $\psi_D$ extends to an isomorphism of $C^*$-algebras as claimed.

The $K$-theory statement follows as any two such isomorphisms $\psi_D,\psi_D'$ differ by conjugation by a unitary multiplier of $C^*_r(\Ga)\otimes\mathcal{K}$, and this induces the identity map on $K$-theory.
\end{proof}

We now specialize to a space of graphs $X=\sqcup G_n$ as in Definition \ref{sg} above, and a covering sequence $\widetilde{X}=(\wgn)_{n\in\N}$ as in Definition \ref{cc}.  For each $n$, let $\pi_n:\wgn\to G_n$ denote the corresponding covering map, and $\Gamma_n$ the group of deck transformations.  The following lemma is a generalization of a fact from \cite{Chen:2008so}.

\begin{philem}\label{philem}
Say $\widetilde{X}$ is an asymptotically faithful covering sequence for $X$.  Then there exists a canonical $*$-homomorphism
$$
\phi:\C[X]\to \frac{\prod\C[\wgn]^{\Gamma_n}}{\oplus\C[\wgn]^{\Gamma_n}}.
$$ 
\end{philem}

\begin{proof}
Let $(T_{x,y})_{x,y\in X}$ be an element of $\C[X]$, and let $S>0$ be such that $T_{x,y}=0$ whenever $d(x,y)>S$.  Let $N$ be such that for all $n,m\geq N$, $d_X(G_n,G_m)\geq 2S$ and $\pi_n:\wgn\to G_n$ is a $2S$-metric cover (such an $N$ exists by choice of the metric on $X$, and the asymptotic faithfulness property).  We may then write 
$$
T=T^{(0)}\oplus\prod_{n\geq N}T^{(n)},
$$
where $T^{(0)}$ is an element of $\mathcal{B}(l^2(G_0\sqcup\cdots\sqcup G_{N-1}))$ and each $T^{(n)}$ is an element of $\mathcal{B}(l^2(G_n))$.  For each $n\geq N$, define an operator $\widetilde{T^{(n)}}\in \C[\wgn]^{\Gamma_n}$ by
$$
\widetilde{T^{(n)}}_{x,y}=\left\{\begin{array}{ll} T^{(n)}_{\pi_n(x),\pi_n(y)} & d(x,y)\leq S \\ 0 & \text{otherwise}\end{array}\right.,
$$
and define $\phi(T)$ to be the image of $\prod_{n\geq N} \widetilde{T^{(n)}}$ under the inclusion-and-quotient map
$$
\prod_{n\geq N}\C[\wgn]^{\Gamma_n}\to \frac{\prod\C[\wgn]^{\Gamma_n}}{\oplus\C[\wgn]^{\Gamma_n}}
$$
(thus $\phi(T)$ does not depend on the choice of $N$).
It is not hard to check that $\phi$ is then a $*$-homomorphism as claimed.  
\end{proof}

\begin{maxcor}\label{maxcor}
If $A$ is any $C^*$-algebraic completion of 
$$
\frac{\prod_{n}\C[\wgn]^{\Gamma_n}}{\oplus_{n}\C[\wgn]^{\Gamma_n}},
$$
then there exists a (unique) $*$-homomorphism $\phi:C^*_{max}(X)\to A$ which extends $\phi$ as in the previous lemma. \qed
\end{maxcor}

We will also be interested in when this map extends to the non-maximal completion $C^*(X)$.  The following definition, due to Chen, Tessera, Wang and the second author \cite{Chen:2008so}, gives a sufficient condition.  The idea of using similar estimates in this context (arising from finite asymptotic dimension of linear type) is due to Higson \cite{Higson:1999km} (see also \cite[Section 9.4]{Roe:2003rw}, which exposits Higson's ideas).

\begin{onl}\label{onl}
Let $X$ be a uniformly discrete metric space.  $X$ is said to have the \emph{operator norm localization property} if there exists a constant $0<c\leq 1$ such that for all $R>0$ there exists $S_R>0$ such that if $T\in \C[X]$ is of propagation at most $R$ then there exists $\xi\in l^2(X,\Hi)$ with diam(supp($\xi))\leq S$ such that $$\|T\xi\|\geq c\|T\|.$$

A collection $\{X_i\}_{i\in I}$ of uniformly discrete metric spaces is said to have the \emph{uniform operator norm localization property} if the constants $c$ and $S_R$ can be chosen to hold across all of the $X_i$ simultaneously.  
\end{onl}

Many natural spaces have the operator norm localization property: for example, it is a theorem of Guentner, Tessera and the second author \cite{Guentner:2008gd} that all countable linear groups have this property.  The following simple lemma will be important in what follows.

\begin{treelem}\label{treelem}
Any family of trees $\{X_i\}$ (say for simplicity with vertices of at most countable valency) has the uniform operator norm localization property.
\end{treelem}

\begin{proof}
Let $X$ be the (unique) tree with all vertices of countably infinite valency.  Then all the $X_i$ embed isometrically in $X$.  The lemma now follows from the fact that $X$ has asymptotic dimension one (cf.\ \cite[Remark 3.2 and Proposition 4.1]{Chen:2008so}).
\end{proof}

The following lemma follows from the definition of the uniform operator norm localization property.

\begin{onllem}\label{onllem}
Let $X$, $\widetilde{X}$, $\phi$ be as in Lemma \ref{philem}, and assume that the sequence of spaces $(\wgn)_{n\in\N}$ has the uniform operator norm localization property.  Then $\phi$ extends to a $*$-homomorphism
$$
\phi:C^*(X)\to \frac{\prod_nC^*(\wgn)^{\Gamma_n}}{\oplus_n C^*(\wgn)^{\Gamma_n}}. \eqno \qed
$$ 
\end{onllem}

\section{Assembly}\label{asssec}

In this section we discuss several versions of the \emph{Baum-Connes assembly map} \cite{Baum:2000ez,Baum:1988sf,Baum:1994pr}, in both the coarse and equivariant settings (a version of these assembly maps was also developed by Kasparov \cite[Section 6]{Kasparov:1988dw}, while the coarse geometric version appears first in work of Roe \cite[Section 6.6]{Roe:1993lq}).  We will give precise definitions where we need them, and refer to the literature otherwise.

Let $X$ be a proper metric space.  Let $K_*(X)$ be the locally finite $K$-homology group of $X$ as defined analytically by Kasparov \cite{Kasparov:1975ht}.  The following definition records our conventions regarding cycles for $K_0(X)$.  The $K_1$ case is similar, and will not be used explicitly in the paper -- see for example \cite[Chapter 8]{Higson:2000bs}.

\begin{k0hom}\label{k0hom}
Let $X$ be a proper metric space as above, and $Z$ be the countable dense subset used to define $\C[X]$.  Noting that the natural representation of $C_0(X)$ on $l^2(Z,\Hi)$ is ample (i.e.\ the representation is non-degenerate and no-non zero element of $C_0(X)$ acts as a compact operator), it follows from \cite[Lemma 8.4.1]{Higson:2000bs} that any element of $K_0(X)$ can be represented by a pair $(F,l^2(Z,\Hi))$ such that
\begin{itemize}
\item $F$ is a bounded operator on $l^2(Z,\Hi)$;
\item $f(1-F^*F)$, $f(1-FF^*)$ are compact operators for all $f\in C_0(X)$;
\item the commutator $[F,f]$ is a compact operator for all $f\in C_0(X)$.
\end{itemize}
The pair $(F,l^2(Z,\Hi))$ is called a \emph{cycle} for $K_0(X)$, and the corresponding equivalence class is denoted $[F,l^2(Z,\Hi)]$.  Often, the space $l^2(Z,\Hi)$ will be left implicit, and we will simply write $F$ for a cycle, and $[F]$ for the corresponding element of $K_0(X)$.
\end{k0hom}

We now define the assembly map in the even dimensional (`$K_0$') case, a process of `taking index' in an appropriately generalized sense.  The odd dimensional (`$K_1$') case can be treated similarly (see for example \cite{Higson:1995fv,Yu:1995bv}), but we will not need the explicit formulas in this paper.

\begin{k0a}\label{k0a}
Let $X$ be a proper metric space as above, and $Z$ be the countable dense subset of $X$ used to define $\C[X]$.  Let $(F_0,l^2(Z,\Hi))$ be a cycle for $K_0(X)$.  Let  $(U_i)_{i\in I}$ be a locally finite, uniformly bounded cover of $X$ and $(\phi_i)$ a subordinate partition of unity.  One readily checks that the operator 
$$
F:=\sum_{i\in I}\sqrt{\phi_i} F_0\sqrt{\phi_i}
$$
(convergence in the strong operator topology) also satisfies the conditions above, and is moreover a multiplier of $\C[X]\subseteq \mathcal{B}(l^2(Z,\Hi))$.  It follows that the matrix 
$$
I(F):=\begin{pmatrix} FF^*+(1-FF^*)FF^* & F(1-F^*F)+(1-FF^*)F(1-F^*F) \\ (1-F^*F)F^* & (1-F^*F) \end{pmatrix}
$$
is an idempotent in the $2\times2$ matrices over the unitzation of $\C[X]$, and is taken to the matrix 
$$
\begin{pmatrix} 1 & 0 \\ 0 & 0 \end{pmatrix}
$$
under the map on $2\times2$ matrices induced by the canonical map from the unitization  of $\C[X]$ to $\C$ (these matrices are part of a standard `index construction' is $K$-theory -- see for example \cite[Chapter 2]{Milnor:1971kl}).  We then define
\begin{align*}
\mu[F]=\mu[F,l^2(Z,\Hi)]: =[I(F)]-\begin{bmatrix} 1 & 0 \\ 0 & 0 \end{bmatrix}
\end{align*}
in $K_0(\C[X])$.  $\mu[F]$ defines an element of $K_0(C^*(X))$ via the inclusion $\C[X]\hookrightarrow C^*(X)$, and it is not hard to check the formula above gives a well-defined homomorphism
$$
\mu:K_0(X)\to K_0(C^*(X)).
$$
\end{k0a}

Combining this with a similar construction in the odd dimensional case defines a homomorphism
$$
\mu:K_*(X)\to K_*(C^*(X))
$$
called the \emph{assembly map}.  There is similarly a \emph{maximal assembly map}
$$
\mu:K_*(X)\to K_*(C_{max}^*(X)),
$$
defined using the inclusion $\C[X]\to C^*_{max}(X)$ -- see \cite[paragraph 4.6]{Gong:2008ja}.

\begin{rips}\label{rips}
Let $X$ be a proper, uniformly discrete metric space.  Let $R>0$.  The \emph{Rips complex of $X$ at scale $R$}, denoted $P_R(X)$, is the simplicial complex with vertex set $X$ and such that a finite set $\{x_1,...,x_n\}\subseteq X$ spans a simplex if and only if $d(x_i,x_j)\leq R$ for all $i,j=1,...,n$.  $P_R(X)$ is then equipped with the \emph{spherical metric} defined by identifying each $n$-simplex with the part of the $n$-sphere in the positive orthant, and equipping $P_R(X)$ with the associated length metric.
\end{rips}

For any $R$ there is a homomorphism $i_R:K_*(C^*(P_R(X)))\to K_*(C^*(X))$ (which need not be an isomorphism in general), coming from the functoriality of $K_*(C^*(\cdot))$ under coarse maps (see for example \cite[Section 4]{Higson:1993th}).  The \emph{coarse assembly map}
$$
\mu:\lim_{R\to\infty}K_*(P_R(X))\to K_*(C^*(X))
$$
is defined to be the limit of the compositions
$$
\xymatrix{ K_*(P_R(X)) \ar[r]^(0.45){\mu_R} & K_*(C^*(P_R(X))) \ar[r]^(0.55){i_R} & K_*(C^*(X)), }
$$
where $\mu_R:K_*(P_R(X))\to K_*(C^*(P_R(X)))$ is the assembly map for $P_R(X)$.  The \emph{coarse Baum-Connes conjecture for $X$} states that this map is an isomorphism, and the \emph{coarse Novikov conjecture for $X$} that it is injective.
There is similarly a \emph{maximal coarse assembly map}
$$
\mu:\lim_{R\to\infty}K_*(P_R(X))\to K_*(C^*_{max}(X));
$$
the \emph{maximal coarse Baum-Connes conjecture for $X$} states that this map is an isomorphism, and the \emph{maximal coarse Novikov conjecture for $X$} that it is an injective.
Note that one does not really need $X$ to be uniformly discrete for the above to make sense, but this is the only case we need, and it helps to simplify definitions slightly.

There is also an equivariant version of assembly.  Let $\Gamma$ be a countable discrete group acting freely and properly by isometries on a proper metric space $X$.  Then there exists a (maximal) \emph{equivariant assembly map}
$$
\mu_{\Gamma}:K_*^{\Gamma}(X)\to K_*(C^*_{(max)}(X)^\Gamma)
$$
where the left hand side is the $\Ga$-equivariant locally finite $K$-homology of $X$.  Note that one does not really need the $\Ga$ action to be free to define $\mu_\Ga$, but this assumption simplifies the definition of $\C[X]^\Ga$, and is satisfied in all the cases we need.  The definition is essentially the same as for the non-equivariant assembly map in Definition \ref{k0a}, except in this case one assumes that the operator $F_0$ is $\Ga$-equivariant (as one may for a proper action), and uses an equivariant partition of unity to form the cut-down $F$.

The equivariant $K$-homology group appearing above is related to non-equivariant $K$-homology by the following lemma, which we will need later.

\begin{indlem}\label{indlem}
Say $X$ is a compact metric space, and $\pi:\widetilde{X}\to X$ a Galois cover with covering group $\Ga$.  Let $Z$ be a countable dense $\Ga$-invariant subset of $X$, and $\widetilde{Z}=\pi^{-1}(Z)$.  Let $[F,l^2(Z,\Hi)]$ be a cycle for $K_0(X)$ as in Definiton \ref{k0a} (the $K_1$ case is similar).

Let $(U_i)_{i=1}^N$ be a finite cover of $X/\Ga$ such that if $\pi:X\to X/\Ga$ is the quotient map then for each $i$, $\pi^{-1}(U_i)$ is equivariantly homeomorphic to $U_i\times\Ga$; having chosen such an identification for each $i$, write $(U_{i,g})_{i=1,g\in\Ga}^N$ for the corresponding equivariant cover of $X$.  

Now, let $(\phi_i)_{i=1}^N$ be a partition of unity subordinate to $(U_i)_{i=1}^N$, and let $(\phi_{i,g})_{i=1,g\in\Ga}^N$ be the corresponding partition of unity subordinate to to the cover $(U_{i,g})$.  
Define for each $(x,y)\in\widetilde{Z}\times\widetilde{Z}$ an element of $\mathcal{B}(\Hi)$ by
\begin{equation}\label{lift}
\widetilde{F}_{x,y}:=\sum_{i=1}^N\sum_{g\in\Ga}\sqrt{\phi_{i,g}(x)}F_{\pi(x),\pi(y)}\sqrt{\phi_{i,g}(y)}.
\end{equation}

Then these expressions define the matrix coefficients of a bounded operator $\widetilde{F}$ on $l^2(\widetilde{Z},\Hi)$.  Moreover, the formula
$$
i^\Ga:[F,l^2(Z,\Hi)]\to [\widetilde{F},l^2(\widetilde{Z},\Hi)]
$$
gives rise to an isomorphism
$$
i^\Ga:K_*(X)\to K^\Gamma_*(\widetilde{X}).
$$
that does not depend on any of the choices involved in its definition. \qed
\end{indlem}
\noindent
The formula in line \eqref{lift} above simply means that we take the operator $F$, cut it down using the partition of unity to get
$$
F'=\sum_{i=1}^k\sqrt{\phi_i}F\sqrt{\phi_i},
$$
and lift each `piece' $\sqrt{\phi_i}F\sqrt{\phi_i}$ to $U_i\times\Ga$ to define $\widetilde{F}$.  The precise formula is helpful in computations.

If $\Gamma$ acts freely, properly and cocompactly on a uniformly discrete metric space $X$, there is a (maximal) \emph{Baum-Connes assembly map}
\begin{equation}\label{bc1}
\mu_{\Gamma}:\lim_{R\to\infty}K_*^\Gamma(P_R(X))\to K_*(C_{(max)}^*(X)^\Gamma),
\end{equation}
defined analogously to the coarse assembly map above; it is also possible to define a version of this homomorphism when the action of $\Ga$ on $X$ is not free, but the definition is slightly more complicated, and we will not need this.  

The more usual definition of the Baum-Connes assembly map \cite{Baum:1994pr} is as a homomorphism  
\begin{equation}\label{bc2}
\mu_{BC}:K_*^{top}(\Gamma)\to K_*(C^*_{(r,max)}(\Gamma)),
\end{equation}
defined using Kasparov's equivariant $KK$-theory \cite{Kasparov:1988dw}.  The \emph{Baum-Connes conjecture for $\Ga$} states that the reduced version of this map is an isomorphism, and the   \emph{strong Novikov conjecture} that it is injective.  We will need the following important lemma: this was folklore for a long time, and a proof was provided by Roe \cite{Roe:2002nx}.

\begin{bclem}\label{bclem}
Say $\Ga$ acts freely, properly, and cocompactly on a proper uniformly discrete metric space $X$.  Then the (maximal) Baum-Connes assembly map $\mu_\Ga$ in line \eqref{bc1} above identifies naturally with the Baum-Connes assembly map from line \eqref{bc2} where we use Lemma \ref{morlem} to identify the right hand sides. \qed
\end{bclem}

\section{Ghost operators}\label{ghostsec}

In this section we collect some facts about a class of `highly non-local' operators that exist in the Roe algebras $C^*(X)$ of certain spaces $X$.  These so-called \emph{ghost operators} were originally introduced by the second author (unpublished).  See for example \cite[Section 11.5.2]{Roe:2003rw} and \cite{Chen:2004bd,Chen:2005cy} for more information on this class of operators.

\begin{ghostdef}\label{ghostdef}
Let $C^*(X)$ be the Roe algebra of a proper metric space $X$, and $Z\subseteq X$ the countable subset of $X$ used to define $\C[X]$.
An operator $T\in C^*(X)$ is said to be a \emph{ghost} if for all $R,\epsilon>0$ there exists a bounded set $B\subseteq X$ such that if $\xi\in l^2(Z,\Hi)$ is supported in $B(x,R)$ for some $x\not\in B$ then $\|T\xi\|<\epsilon$. 
\end{ghostdef} 

\begin{gbgrem}\label{gbgrem}
If $X$ is uniformly discrete and of bounded geometry, $T$ is a ghost if and only if 
$$
\lim_{(x,y)\to\infty \text{ in } Z\times Z}T_{x,y}=0.
$$
\end{gbgrem}

Clearly a compact operator on $l^2(X)$ is a ghost, and if $X$ has \emph{property A}, the converse is true: see \cite[Proposition 11.43]{Roe:2003rw}, \cite{Chen:2004bd} and \cite[Section 4]{Chen:2005cy}.  Non-compact ghost operators can exist, however -- we give examples below.  They are very mysterious objects: non-compact ghost operators have a definite global existence (as non-compact), while simultaneously being `locally almost invisible'.  On the other hand, the coarse Baum-Connes conjecture predicts that the (a priori, global) $K$-theory group $K_*(C^*(X))$ can be modeled using the local information in the $K$-homology groups $K_*(P_R(X))$.  There is thus some tension between a class in $K_*(C^*(X))$ being represented by a non-compact ghost, and its being in the image of the coarse assembly map; we exploit this tension to get counterexamples to the surjectivity part of the coarse Baum-Connes conjecture in Section \ref{surjsec} below.  Ideas along these lines are originally due to Higson \cite{Higson:1999km}.

We will be interested in the following examples of ghost operators: in all cases the operators are infinite rank projections, so not compact.

\begin{ghostex}\label{ghostex}
\begin{enumerate}[(i)]
\item Let $X=\sqcup G_n$ be a space of graphs.  Let $\Delta_n$ be the graph Laplacian on $G_n$, and fix a rank one projection $q\in \mathcal{K}$.  Let  
$$\Delta:=\prod_n(\Delta_n\otimes q)$$
be the Laplacian on $l^2(X)$ tensored by the projection $q$; note that $\Delta$ is an element of $\C[X]$ (and in fact, of propagation one).  

Say now that $X$ is an expander, so that $\Delta\in \C[X]$ has spectrum contained in $\{0\}\cup [c,2]$ for some $c>0$.  Let
$$
p=\lim_{t\to\infty}e^{-t\Delta}
$$  
(the limit exists in the norm topology using the `spectral gap' of $\Delta$) be the spectral projection associated to $0\in\text{spectrum}(\Delta)$.  Then $p=\prod p^{(n)}\in C^*(X)$, where $p^{(n)}\in\mathcal{B}(l^2(G_n))$ is the projection onto the constant functions in $l^2(G_n)$.  It is not hard to see that $p^{(n)}$ is given by the formula
$$
p^{(n)}_{x,y}=\frac{1}{|G_n|}
$$
for all $n$ and $x,y\in G_n$; as $|G_n|\to\infty$ as $n\to\infty$ by assumption, $p$ is a ghost, at least in the bounded geometry case.  Note that $p$ is an infinite rank projection, however, so non-compact.  In the non-bounded geometry case, $p$ will still be a ghost if the diameter of $G_n$ increases suitably quickly with respect to the speed at which the vertex degrees of the $G_n$ increase, but this is not so important for us (cf.\ however Lemma \ref{kazdies} below).  We call $p$ the \emph{basic Kazhdan projection} associated to $X$.

\item Say $\Gamma$ is a finitely generated group equipped with the edge metric coming from the Cayley graph associated to some fixed finite generating set.  Assume that $(\Gamma_n)_{n\in\N}$ is a nested  sequence of finite index, normal subgroups such that $\cap \Gamma_n=\{e\}$ and with respect to which $\Ga$ has property $(\tau)$.  Let $X$ be the space of graphs $\sqcup \Ga/\Ga_n$, where each $\Ga/\Ga_n$ is given the Cayley graph structure associated to the image of the fixed generating set of $\Ga$.  Let 
$$\sigma:\Ga\to \mathcal{U}(\C^{\text{dim}(\sigma)})$$
be any finite dimensional irreducible representation of $\Gamma$ that factors through $\Ga/\Ga_N$ for some $N$ (which we assume is the smallest such), whence through $\Ga/\Ga_n$ for all $n\geq N$.  

Note that the group algebra $\C[\Ga]$ admits a faithful representation on $l^2(X)$ by right multiplication, and that this representation includes $\C[\Ga]$ as a subalgebra of $C^*(X)$.  Let $C^*_{X}(\Ga)$ be the closure of the group algebra in $C^*(X)$, and note that property $(\tau)$ implies that $\sigma$ is isolated in the spectrum of $C^*_{X}(\Ga)$ (this is true of any irreducible finite-dimensional representation that factors through some $\Ga/\Ga_n$, by essentially the same proof as in the property (T) case -- see \cite[Theorem 1.2.5]{Bekka:2000kx}).  Hence there exists a projection $p_{\sigma}\in C^*_{X}(\Gamma)\subseteq C^*(X)$ which has image the $\sigma$-isotypical component of $l^2(X)$, i.e. the $\Ga$-invariant subspace of $l^2(X)$ defined as the sum of all $\Ga$-invariant subspaces of $l^2(X)$ that are unitarily equivalent (as $\Ga$-representations) to $\sigma$. 

We then have that $p_\sigma=\oplus_{n\geq N} p^{(n)}_\sigma$, where $p^{(n)}_\sigma\in\mathcal{B}(l^2(\Gamma/\Gamma_n))$ is the projection onto the $\sigma$-isotypical component of $l^2(\Gamma/\Gamma_n)$.  Let $\chi_\sigma:\Gamma\to\C$ be the character associated to $\sigma$.  Letting $[g]\in\Gamma/\Gamma_n$ denote the image of $g\in\Gamma$, the matrix coefficients of $p^{(n)}_\sigma$ are given by
$$
(p^{(n)}_\sigma)_{[g],[h]}=\left\{\begin{array}{ll} 0 & n<N \\ \frac{\text{dim}(\sigma)}{|\Ga/\Ga_n|}\chi_\sigma(g^{-1}h) & n\geq N \end{array}\right.,
$$
as follows from basic facts in the representation theory of finite groups.  In particular $p_\sigma$ is a ghost operator, and also an infinite rank projection.  There will be countably infinitely many such $\sigma$s for any $X$ built out of a property $(\tau)$ group $\Ga$ as above; we call $p_\sigma\in C^*(X)$ the \emph{Kazhdan projection} associated to $\sigma$.  

Note that in this special case, the basic Kazhdan projection from part (i) above is the same as the Kazhdan projection $p_\textbf{1}$ associated to the trivial representation $\textbf{1}:\Ga\to\C$ as in this part.
\end{enumerate}
\end{ghostex}

It would be interesting to have other concrete examples of ghost operators, especially if they came from natural geometric (as opposed to representation theoretic) hypotheses. We leave finding others as a problem.

\begin{ghostprob}\label{ghostprob}
Find more natural geometric examples of ghost projections (or more general operators).  Elucidate the structure of the ghost operators (which form an ideal in $C^*(X)$); for this latter problem, cf.\ \cite{Chen:2004bd,Chen:2005cy,Wang:2007uf}.
\end{ghostprob}

The following lemma says that a suitable covering sequence of a space of graphs does not `see' any ghost operators.

\begin{ghostlem}\label{ghostlem}
Let $X=\sqcup G_n$ be a space of graphs.  Assume that $\widetilde{X}=(\wgn)_{n\in\N}$ is an asymptotically faithful covering sequence for $X=\sqcup G_n$, such that $\widetilde{X}$ also has the uniform operator norm localization property.  Let  
$$
\phi:C^*(X)\to \frac{\prod_nC^*(\wgn)^{\Gamma_n}}{\oplus_n C^*(\wgn)^{\Gamma_n}}
$$
be the map from Lemma \ref{onllem}.  Then $\phi(T^G)=0$ for any ghost operator $T^G$.
\end{ghostlem}

\begin{proof}
Fix $\epsilon>0$.  Let $T^G$ be a ghost operator, and choose $T$ of finite propagation such that $\|T^G-T\|<\epsilon$; say the propagation of $T$ is $R$.  Let $\widetilde{T^{(n)}}$ be as in the construction of $\phi(T)$ in Lemma \ref{philem} (so $\widetilde{T^{(n)}}$ exists for all $n$ suitably large), and note that each $\widetilde{T^{(n)}}$ has propagation most $R$.  Hence, using the uniform operator norm localization property, there exist $S_R>0, c>0$ and for each $n$ some $\widetilde{\xi_n}\in l^2(\wgn)$ of norm one and support diameter at most $S_R$ such that 
$$
\|\widetilde{T^{(n)}}\widetilde{\xi_n}\|\geq c\|\widetilde{T^{(n)}}\|. 
$$
On the other hand, using the covering faithfulness property (with respect to the paramter $S_R$), for all $n$ suitably large there exists norm one $\xi_n\in l^2(G_n)$ such that $\|\widetilde{T^{(n)}}\widetilde{\xi_n}\|=\|T\xi_n\|$.  Using what we have so far and the ghost property of $T^G$, for all $n$ suitably large
$$
c\|\widetilde{T^{(n)}}\|\leq \|T^{(n)}\xi_n\|\leq \|T^G-T\|+\|T^G\xi_n\|<2\epsilon.
$$
Hence 
$$
\|\phi(T^G)\|\leq \epsilon+\|\phi(T)\|\leq\epsilon+\limsup_n\|\widetilde{T^{(n)}}\|<\epsilon+\frac{2\epsilon}{c};
$$
as $\epsilon$ was arbitrary, and $c$ independent of $\epsilon$, this completes the proof.
\end{proof}

The following lemma shows that the basic Kazhdan projection as in Example \ref{ghostex} (i) always maps to $0$ under $\phi$, even in the non-bounded geometry case when it need not be a ghost.  We will also need this when we study $C^*_{max}(X)$ in Section 7 in the second part of this series \cite{Willett:2010zh}, as the definition of ghost operators does not make sense in $C^*_{max}(X)$.

\begin{kazdies}\label{kazdies}
Assume that $X=\sqcup G_n$ is a weak expander and $\widetilde{X}=(\wgn)_{n\in\N}$ an asymptotically faithful covering sequence with the uniform operator norm localization property.  Assume moreover that all but finitely many of the graphs $\wgn$ are infinite.
Let $p\in C^*(X)$ be the basic Kazhdan projection as in Examples \ref{ghostex} (i). Then the image of $p$ under $\phi$ is zero.
\end{kazdies}

\begin{proof}
Consider $\Delta\in C^*(X)$, and its image 
$$
\phi(\Delta)\in \frac{\prod C^*(\wgn)^{\Gamma_n}}{\oplus C^*(\wgn)^{\Gamma_n}}.
$$
Note that $\text{spectrum}(\phi(\Delta))\subseteq \{0\}\cup[c,2]$ for some $c>0$, as this is true for $\Delta$ itself.  Note moreover, however, that if $0$ were in $\text{spectrum}(\Delta)$, then the associated spectral projection would necessarily be of the form
$$\tilde{p}=[\prod p_n]\in \frac{\prod C^*(\wgn)^{\Gamma_n}}{\oplus C^*(\wgn)^{\Gamma_n}},$$
where each $p_n\in C^*(\wgn)^\Ga_n\subseteq \mathcal{B}(l^2(\wgn))$ is the projection onto the constant functions in $l^2(\wgn)$.  $p_n$ is thus zero for all but finitely many $n$, as all but finitely many of the $\wgn$ are infinite, and thus $\tilde{p}$ itself is zero.  Finally, then, we have that
$$\phi(p)=\phi(\lim_{t\to\infty}e^{-t\Delta})=\lim_{t\to\infty}e^{-t\phi(\Delta)}=\tilde{p}=0$$
as claimed.
\end{proof}

\section{Surjectivity of the coarse assembly map}\label{surjsec}

In this section, we use a version of the Atiyah-$\Ga$-index theorem \cite{Atiyah:1976th} to show that the coarse Baum-Connes assembly map fails to be surjective for certain classes of expanders.  The idea is due to Higson \cite{Higson:1999km}. We outline the basic argument below.

Let $X=\sqcup G_n$ be a space of graphs.  Then $C^*(X)$ is equipped with a $*$-homomorphism
$$
d:C^*(X)\to \frac{\prod \mathcal{K}(l^2(G_n,\Hi))}{\oplus\mathcal{K}(l^2(G_n,\Hi))},
$$
inducing a homomorphism
\begin{equation}\label{detect}
d_*:K_0(C^*(X))\to K_0\Big(\frac{\prod \mathcal{K}(l^2(G_n,\Hi))}{\oplus\mathcal{K}(l^2(G_n,\Hi))}\Big)\cong\frac{\prod\Z}{\oplus\Z}
\end{equation}
that we will use to detect non-zero $K$-theory classes.  If $[p]$ is a class in $K_0(C^*(X))$ such that $p$ decomposes as $p=\oplus p_n$, $p_n\in  \mathcal{K}(l^2(G_n,\Hi))$ a projection, then $d_*[p]$ is of course very concrete: it is just
$$
[\text{dim}(p_0),\text{dim}(p_1),\text{dim}(p_2),...]\in \frac{\prod\Z}{\oplus\Z}.
$$
For example, if $p$ is the basic Kazhdan projection from Example \ref{ghostex} (i), then $d_*[p]=[1,1,1,...]$, while if $p_\sigma$ is one of the Kazhdan projections from Example \ref{ghostex} (ii), then 
$$d_*[p_\sigma]=[0,0,...,0,\text{dim}(\sigma)^2,\text{dim}(\sigma)^2,\text{dim}(\sigma)^2,...],$$
where the zeros persist until the first quotient $\Ga/\Ga_n$ through which $\sigma$ factors, and after that the sequence is constant.
In particular, $d$ can be used to detect the non-triviality of the $K_0$-classes defined by these projections.

The underlying idea of this section is to use a version of Atiyah's $\Gamma$-index theorem \cite{Atiyah:1976th} to compare the map $d_*$ with the composition of 
$$
\phi_*:K_0(C^*(X))\to K_0\Big(\frac{\prod_nC^*(\wgn)^{\Gamma_n}}{\oplus_n C^*(\wgn)^{\Gamma_n}}\Big)
$$ 
from Lemma \ref{onllem} and a certain trace-like map
$$
T:K_0\Big(\frac{\prod_nC^*(\wgn)^{\Gamma_n}}{\oplus_n C^*(\wgn)^{\Gamma_n}}\Big)\to\frac{\prod\R}{\oplus\R}
$$ 
to show that such projections cannot be in the image of the coarse assembly map (this approach is originally due to Higson \cite{Higson:1999km}).  Indeed, we will show that for elements $[x]\in K_*(P_R(X))$, there is an identity
$$
d_*(\mu[x])=T(\phi_*(\mu[x]))
$$
(this is an abstract version of Atiyah's $\Ga$-index theorem \cite{Atiyah:1976th}), where we consider 
$$
d_*(\mu(x))\in\frac{\prod\Z}{\oplus\Z}\subseteq\frac{\prod\R}{\oplus\R}
$$
to make sense of this.  On the other hand, for the basic Kazhdan projection $p$ discussed above $d_*[p]\neq 0$, while Lemma \ref{ghostlem} implies that $T(\phi_*[p])=0$; this immediately implies that the element $[p]\in K_0(C^*(X))$ cannot be in the image of $\mu$, and the same argument shows that this is also true for the elements $[p_\sigma]\in K_0(C^*(X))$ defined by the Kazhdan projections $p_\sigma$.  In general, the following theorem is the main result of this section.  

\begin{negthe}\label{negthe}
Say that $X=\sqcup G_n$ is a space of graphs.  Say that there exists an asymptotically faithful covering sequence $\widetilde{X}=(\wgn)_{n\in\N}$ with the uniform operator norm localization property.

Then if $[p]\in K_0(C^*(X))$ is the class of a non-compact ghost projection $p\in C^*(X)$, $[p]$ is not in the image of the coarse assembly map.
\end{negthe}
 
In particular, of course, if there are any non-compact ghost projections in $C^*(X)$, where $X$ satisfies the assumptions in the theorem, then the coarse assembly map is not surjective.  For the sake of concreteness, note the following two corollaries.

\begin{negcor}\label{negcor}
\begin{enumerate}[(i)]
\item Say $X=\sqcup G_n$ is an expander with large girth.  Then the class $[p]\in K_0(C^*(X))$ of the basic Kazhdan projection is not in the image of the coarse assembly map.
\item Say $X=\sqcup\Gamma/\Gamma_n$, where $\Gamma$ is a finitely generated discrete group with the operator norm localization property, and $(\Gamma_n)_{n\in\N}$ is a nested sequence of finite index normal subgroups such that $\cap \Ga_n=\{e\}$ and so that $\Ga$ has property $(\tau)$ with respect to this sequence.  Then none of the classes $[p_\sigma]\in K_0(C^*(X))$ as in Example \ref{ghostex} are in the image of the coarse assembly map.
\end{enumerate}
\end{negcor}

\begin{proof}
The first claim follows from Example \ref{cfex} (i), Lemma \ref{treelem} and Example \ref{ghostex} (i).  The second follows from Example \ref{cfex} (ii) and Example \ref{ghostex} (ii).
\end{proof}

This argument leaves open the possibility that for $\Ga$ a property $(\tau)$ group, the classes of the various $[p_\sigma]$ are all the same in $K_0(C^*(X))$, up to integer multiples.  Note that they \emph{are} genuinely different in $K_0(C^*_X(\Ga))$, generating an infinite rank subgroup.  We leave this as a problem.

\begin{taupprob}\label{taupprob}
Are the various classes $[p_\sigma]$ in $K_0(C^*(X))$ `genuinely different'?
\end{taupprob}

The next result is an abstract version of Atiyah's $\Gamma$-index theorem \cite{Atiyah:1976th}.  Before we state it, note that for any compact metric space $Y$, the Roe algebra $C^*(Y)$ is isomorphic to an abstract copy $\mathcal{K}$ of the compact operators.  It thus has a canonical unbounded trace $Tr:C^*(Y)\to\C\cup\{\infty\}$ which gives rise to a map on $K$-theory denoted
$$Tr_*:K_0(C^*(Y))\to\R$$
(this map does not depend on the choice of isomorphism $C^*(Y)\cong\mathcal{K}$).  Say also that $\widetilde{Y}$ is a Galois cover of $Y$ with covering group $\Ga$.  Then $C^*_r(\Ga)$ has a canonical trace defined by `taking the coefficient of the identity', i.e.\
$$
\sum_{g\in\Ga}\lambda_gu_g\mapsto \lambda_e,
$$
whence $C^*(\widetilde{Y})^\Ga$, which is isomorphic to $C_r^*(\Ga)\otimes \mathcal{K}$ by Lemma \ref{morlem}, has an unbounded trace $\tau:C^*(\widetilde{Y})^\Ga\to\C\cup\{\infty\}$ defined by taking the tensor product of the traces on $C^*_r(\Ga)$ and $\mathcal{K}$; this defines a map on $K$-theory which we denote
$$
\tau_*:K_0(C_r^*(\widetilde{Y})^\Ga)\to \R
$$
(again, the choice of isomorphism $C^*(\widetilde{Y})^\Ga\cong C^*_r(\Ga)\otimes\mathcal{K}$ does not affect this homomorphism).  Note that in the above we are tacitly using that the domain of each of these unbounded traces is a holomorphically closed subalgebra of the relevant $C^*$-algebra to get a map defined on the entire $K$-theory group.

\begin{atiyahthe}\label{atiyahthe}
Let $Y$ be a finite CW complex, and $\pi:\widetilde{Y}\to Y$ a Galois covering space, with group of deck transformations $\Ga$.  

Then for any $[x]\in K_0(Y)$
$$
Tr_*(\mu[x])=\tau_*(\mu_\Ga(i^\Ga[x]),
$$  
where $\mu:K_*(Y)\to K_*(C^*(Y))$ is the assembly map, $i^\Ga:K_*(Y)\to K^\Ga_*(\widetilde{Y})$ is the $K$-homology induction isomorphism, and $\mu_\Ga:K_*^\Ga(\widetilde{Y})\to K_*(C^*(\widetilde{Y})^\Ga)$ is the equivariant assembly map.
\end{atiyahthe}

We give a short proof using the Baum--Douglas model of $K$-homology \cite{Baum:1980pt}, and for the reader's convenience a different proof that avoids reference to Atiayh's work \cite{Atiyah:1976th} in the classical case. 

\begin{proof}
Using the Baum--Douglas geometric model of $K$-homology as recently studied by Baum--Higson--Schick \cite{Baum:2007ek}, $[x]$ can be represented by a cycle $(M,E,f)$ where $M$ is a compact spin$^\text{c}$ manifold, $E$ is a complex vector bundle over $M$, and $f:M\to Y$ is a continuous map.  The image of $[x]$ under $i^\Ga:K_*(Y)\to K_*^\Ga(\widetilde{Y})$ is then represented by the cycle $(\widetilde{M},\widetilde{E},\widetilde{f})$, where $\widetilde{M}$ is the principal $\Ga$-bundle over $M$ defined as the pullback of $\widetilde{Y}$ along $\phi$, and $\widetilde{E}$, $\widetilde{f}$ are the corresponding lifts (see \cite{Baum:2009hq} for the Baum--Douglas geometric description of equivariant $K$-homology).  Let $D$, $\widetilde{D}$ be the spin$^\text{c}$ Dirac operators on $M$, $\widetilde{M}$ respectively, and $[D]\in K_*(M)$, $[\widetilde{D}]\in K_*^\Ga(\widetilde{M})$ the corresponding classes in Kasparov $K$-homology.  It follows from the definition of the map from Baum-Douglas $K$-homology to Kasparov $K$-homology and naturality of the assembly map that
$$
Tr_*(\mu[x])=Tr_*(\mu(f_*[D]))=Tr_*(\mu[D])
$$
and 
$$
\tau_*(\mu_\Ga(i^\Ga[x])=\tau_*(\mu_\Ga(f_*[\widetilde{D}]))=\tau_*(\mu_\Ga[\widetilde{D}]);
$$
thus we have reduced the original problem to showing that 
$$
Tr_*(\mu[D])=\tau_*(\mu_\Ga[\widetilde{D}]).
$$
This, however, is equivalent to the classical statement of Atiyah's $\Ga$-index theorem for the Dirac-type operator $D$ \cite{Atiyah:1976th}.
\end{proof}  

\begin{proof}[A different proof of Theorem \ref{atiyahthe}]
The heat kernel method underlies Atiyah's original proof of the $\Ga$ index theorem  \cite{Atiyah:1976th}.  For the benefit of readers unfamiliar with the heat kernel method, we sketch a proof below of Theorem \ref{atiyahthe} that avoids its use and is closer to the rest of this paper.  Note, however, that the central ideas -- localization and lifting -- of both the proof below and of the heat kernel method for proving the $\Ga$-index theorem are the same.  This is the only place in this paper where we use a notion of propagation defined for operators on a Hilbert space not of the form $l^2(Z,\Hi)$; we refer the reader to (for example) \cite[Chapter 6]{Higson:2000bs} for details.

Let $[x]$ be a class in $K_*(Y)$.  Using the aforementioned results of Baum-Higson-Schick, $[x]$ can be represented by $f_*[D]$ for some Dirac-type operator $D$ on a compact spin$^\text{c}$ manifold $M$ and continuous map $f:M\to Y$.  Concretely, $f_*[D]$ is represented on a Hilbert space $L^2(M,S)$ of $L^2$-sections of a bundle $S$ over $M$ on which $D$ acts, which we consider as a module over $C(Y)$ using $f$.  Using a finite propagation speed argument due to Roe (see for example \cite{Roe:1989sh}, particularly Lemma 7.5) we may assume that the class $f_*[D]\in K_*(Y)$ is represented by an operator $F$ on $L^2(M,S)$ of arbitrarily small propagation (for the metric on $Y$ -- this uses that the map $f:M\to Y$ is uniformly continuous) and such that $FF^*-1$ and $F^*F-1$ are in $S_1(L^2(M,S))$, the Banach-$*$ algebra of  trace class operators on $L^2(M,S)$; concretely, $F=\chi(D)$, where $\chi$ is a \emph{chopping function} with suitably good properties as in \cite[Lemma 7.5]{Roe:1989sh}.   By adding a degenerate module of the form (for example) $(1,l^2(Z,\Hi))$, $Z$ a countable dense subset of $Y$, we may assume that the module on which this cycle is defined is ample for $C(Y)$ (i.e.\ the representation of $C(Y)$ is unital and no non-zero element of $C(Y)$ acts by a compact operator -- this is useful to identify the constructions below with the usual assembly maps); by abuse of notation, denote this new cycle by $(F, \mathcal{H})$, noting that the new $F$ still has the same propagation and trace properties as the old one.

We will now define a lift $(\phi^L(F),l^2(\Ga)\otimes\mathcal{H})$ of the cycle $(F,\mathcal{H})$ to $\widetilde{Y}$.
Fix a precompact Borel fundamental domain $D\subseteq \widetilde{Y}$: precisely, we require that $\widetilde{Y}=\sqcup_{g\in\Ga}g\cdot D$, and that for each $g\in\Ga$, $\pi|_{g\cdot D}:g\cdot D\to Y$ is a Borel isomorphism.  For each $g\in\Ga$ and $f\in C_0(\widetilde{Y})$ note that $f|_{g\cdot D}\circ(\pi|_{g\cdot D})^{-1}$ is a bounded Borel function on $Y$; having extended the representation of $C(Y)$ on $\mathcal{H}$ to a representation of the bounded Borel functions on $Y$, we may define a representation of $C_0(\widetilde{Y})$ on $l^2(\Ga)\otimes\mathcal{H}$ by 
$$
f\cdot (\delta_g\otimes \xi ):=\delta_g\otimes ( f|_{g\cdot D}\circ(\pi|_{g\cdot D})^{-1})\xi
$$
for all $f\in C_0(\widetilde{Y})$, $g\in \Ga$ and $\xi\in\mathcal{H}$.  This gives the $C_0(\widetilde{Y})$-module part of our lift of $(F,\mathcal{H})$.

Now, let $\epsilon>0$ be such that $\pi:\widetilde{Y}\to Y$ is an $\epsilon$-metric cover, and assume from now on that the propagation of $F$ is less than $\epsilon/10$.   Let $(U_i)_{i=1}^N$ be a finite Borel cover of $Y$ by disjoint sets of diameter less than $\epsilon/2$, and let $\chi_i$ be the characteristic function of $U_i$.   For any bounded operator $T$, define 
$$
T_{i,j}:=\chi_iT\chi_j:\chi_j\mathcal{H}\to\chi_i\mathcal{H};
$$
in this way $T$ can be represented by an $N\times N$ matrix $(T_{i,j})_{i,j=1}^N$, which will help us lift $T$ to $\widetilde{Y}$.  Let $\widetilde{U_i}=\pi^{-1}(U_i)$, which splits as a disjoint collection $\widetilde{U_i}=\sqcup_{g\in\Ga}U_{i,g}$, where each $U_{i,g}$ is isometric to $U_i$ via $\pi$.  Let $\chi_{i,g}$ be the characteristic function of $U_{i,g}$. 
Note that any operator $T$ on $l^2(\Ga)\otimes\mathcal{H}$ can be written uniquely as a matrix $(T_{(i,g),(j,h)})_{i,j=1,g,h\in\Ga}^N$, where
$$
T_{(i,g),(j,h)}:\chi_{j,h}(l^2(\Ga)\otimes\mathcal{H})\to\chi_{i,g}(l^2(\Ga)\otimes\mathcal{H}).
$$
For each ${i,g}$ there are canonical identifications 
\begin{align*}
\chi_{i,g}(l^2(\Ga,\mathcal{H})) & =\bigoplus_{h\cdot D\cap U_{i,g}\neq\emptyset}\delta_h\otimes (\chi_{i,g}|_{h\cdot D}\circ (\pi|_{h\cdot D})^{-1})\mathcal{H} \\
& =\bigoplus_{h\cdot D\cap U_{i,g}\neq\emptyset} \chi_i|_{\pi(h\cdot D\cap U_{i,g})}\mathcal{H} \\
& =\chi_i\mathcal{H}.
\end{align*}
Using these identifications, for any operator 
$$T_{i,j}:\chi_j\mathcal{H}\to \chi_i\mathcal{H}$$
and any $g,h\in\Ga$ we may form a lift
$$
\widetilde{T_{i,j}}:\chi_{j,h}(l^2(\Ga,\mathcal{H}))\to \chi_{i,g}(l^2(\Ga,\mathcal{H})).
$$

Denote now by $\mathcal{L}_{\epsilon/2}[Y]$ the collection of (not necessarily locally compact) operators on $\mathcal{H}$ of propagation at most $\epsilon/2$, and similarly, denote by $\mathcal{L}_{\epsilon/2}[\widetilde{Y}]^\Ga$ the $\Ga$-invariant operators on $l^2(\Ga)\otimes\mathcal{H}$ of propagation at most $\epsilon/2$.  For any $T\in \mathcal{L}_{\epsilon/2}[Y]$, define a lifted operator $\phi^L(T)$ by the matrix coefficient formula
$$
\phi^L(T)_{(i,g),(j,h)}=\left\{\begin{array}{ll} \widetilde{T_{i,j}} & d(U_{i,g},U_{j,h})<\epsilon/2 \\ 0 & \text{otherwise}\end{array}\right.,
$$
and note that $\phi^L(T)$ is an element of $\mathcal{L}_{\epsilon/2}[\widetilde{Y}]^\Ga$.

This defines a `lifting map' 
$$\phi^L:\mathcal{L}_{\epsilon/2}[Y]\to \mathcal{L}_{\epsilon/2}[\widetilde{Y}]^\Ga,$$
which is in fact a $*$-homomorphism whenever multiplication makes sense in $\mathcal{L}_{\epsilon/2}[Y]$ ($\phi^L$ is a close `local' analogoue of the map $\phi$ from Lemma \ref{philem} -- `$^L$' stands for `local').  In particular, then, there exists a lift $\phi^L(F)$ of $F$ to $l^2(\Ga)\otimes\mathcal{H}$, which can be used to define the $K$-homology class $i^\Ga[x]$, i.e.\ there is an equality
$$[\phi^L(F),l^2(\Ga)\otimes\mathcal{H}]=i^\Ga[x]\in K_*^\Ga(\widetilde{Y})$$
of $K$-homology classes (cf.\ the description of $i^\Ga$ in Lemma \ref{indlem} above).  This class is the lift of $[F,\mathcal{H}]$ that we have been trying to construct.  
The reason this lift exists is that the operator $F$ representing the class $[x]\in K_*(Y)$ can be taken to be `arbitrarily local' (unlike, for example, a non-compact ghost operator).

Let now $S_1$ denote the algebra of trace class operators on $\mathcal{H}$ and $\C[\Ga]\odot S_1$ the algebraic tensor product of the group algebra of $\Ga$ and $S_1$, which is represented on $l^2(\Ga)\otimes \mathcal{H}$ in the obvious way.  Using the formulas from Definition \ref{k0a} above, we may define `index operators' $I(F)$ and $I(\phi^L(F))$, which, up to taking $2\times 2$ matrices, give operators in the unitizations of $S_1$, $\C[\Ga]\odot S_1$ respectively; this uses the assumption that $1-F^*F$ and $1-FF^*$ are trace class from the first paragraph above. 

Write $Tr:S_1\to \C$ and $\tau:\C[\Ga]\odot S_1\to\C$ for the traces on these algebras, defined precisely analogously to those used in the statement of the theorem.  It follows from the definitions of the assembly maps and the comments so far that
\begin{equation}\label{ac}
Tr_*(\mu[x])=Tr\Big(I(F)-\begin{pmatrix} 1 & 0 \\ 0 & 0 \end{pmatrix}\Big) ~~\text{ and } \tau_*(\mu_\Ga(i^\Ga[x]))=\tau\Big(I(\phi^L(F))-\begin{pmatrix} 1 & 0 \\ 0 & 0 \end{pmatrix}\Big);
\end{equation}
in both cases, the left hand side is the image of a certain class under a trace map on $K$-theory, while the right hand side is a concrete operator trace.  Now, using the fact that $\phi^L$ is a $*$-homomorphism on its domain (and remains so when extended to $2\times 2$ matrices), and that $F$ has propagation less than $\epsilon/10$, $I(F)$ has propagation less than $\epsilon/2$, so that $\phi^L(I(F))$ both makes sense and is equal to $I(\phi^L(F))$.  Looking back at line \eqref{ac} above, then, it thus suffices for the proof of the theorem to show that
$$
Tr\Big(I(F)-\begin{pmatrix} 1 & 0 \\ 0 & 0 \end{pmatrix}\Big)=\tau\Big(\phi^L(I(F))-\begin{pmatrix} 1 & 0 \\ 0 & 0 \end{pmatrix}\Big).
$$  

To complete the proof, write 
$$
\phi^L(I(F))=\sum_{g\in\Ga}u_g\odot k_g\in M_2((\C[\Ga]\odot S_1)^+)
$$
(`$^+$' denotes `unitization') where each $k_g$ is in $M_2(S_1^+)$, and note that, by definition of $\tau$, 
$$\tau\Big(\phi^L(I(F))-\begin{pmatrix} 1 & 0 \\ 0 & 0 \end{pmatrix}\Big)=Tr\Big(k_e-\begin{pmatrix} 1 & 0 \\ 0 & 0 \end{pmatrix}\Big);$$ 
thus it suffices to show that
$$
Tr\Big(I(F)-\begin{pmatrix} 1 & 0 \\ 0 & 0 \end{pmatrix}\Big)=Tr\Big(k_e-\begin{pmatrix} 1 & 0 \\ 0 & 0 \end{pmatrix}\Big).
$$
Up to identifying the fundamental domain $D$ and $Y$ itself, however, it is not hard to use the description of $\phi^L$ above to check that $I(F)=k_e$, and we are done. \end{proof}

Now, let $\widetilde{X}$ be as in the statement of Theorem \ref{negthe}.  Fix a basepoint $b_n\in \wgn$ for each $n$, and let $\Gamma_n$ be the group of deck transformations of $\wgn$ so that $\wgn/\Gamma_n=G_n$ (having chosen $b_n$, $\Gamma_n$ is of course unique).  For each $n$ consider the homomorphism 
$$\tau^{(n)}_*:K_0(C^*(\wgn)^{\Gamma_n})\to\R$$
defined using the (unbounded) trace discussed above. We may thus define a group homomorphism
$$T=\frac{\prod\tau^{(n)}_*}{\oplus\tau^{(n)}_*}:K_0\Big(\frac{\prod_{n} C^*(\wgn)^{\Gamma_n}}{\oplus_{n} C^*(\wgn)^{\Gamma_n}}\Big)\to\frac{\prod_n\R}{\oplus_n\R}.$$
Moreover, using the fact that $\widetilde{X}$ is an asymptotically faithful covering sequence for $X$ and the uniform operator norm localization property, there is a $*$-homomorphism
$$\phi:C^*(X)\to \frac{\prod_{n} C^*(\wgn)^{\Gamma_n}}{\oplus_{n} C^*(\wgn)^{\Gamma_n}}$$
as in Lemma \ref{onllem}. 

The following lemma, combined with Lemma \ref{ghostlem}, essentially completes the proof of Theorem \ref{negthe}.

\begin{atlem}\label{atlem}
If $p$ is a projection in $C^*(X)$ such that the class $[p]\in K_0(C^*(X))$ is in the image of the coarse assembly map 
$$\mu:\lim_{R\to\infty} K_0(P_R(X))\to K_0(C^*(X))$$
then 
$$T(\phi_*[p]))=d_*[p]\in\frac{\prod_n\R}{\oplus_n\R}$$
(here $\phi$ is as in Lemma \ref{philem}).
\end{atlem}

\begin{proof}
Fix $R>0$.  It follows from the definition of the Rips complex and the metric on $X$ that
$$P_R(X)=P_0\sqcup\bigsqcup_{n\geq N_R}P_R(G_n)$$
for some $P_0$ and $N_R\in \N$; assume moreover that $N_R$ is so large that $\Ga_n$ acts freely and properly on $P_R(\wgn)$ and
$$P_R(\wgn)/\Ga_n=P_R(G_n)$$
for all $n\geq N_R$ (this is possible by the asymptotic faithfulness property).  In particular, for all $n\geq N_R$, $K_*(P_R(X))$ admits a product decomposition 
\begin{equation}\label{rsplit}
K_*(P_R(X))=K_*(P_0)\oplus \prod_{n\geq N_R}K_*(P_R(G_n)).
\end{equation} 
and there is an induction isomorphism
$$
i^{\Ga_n}:K_*(P_R(G_n))\stackrel{\cong}{\to}K^{\Gamma_n}_*(P_R(\wgn))
$$
as in Lemma \ref{indlem}.
Define $\widetilde{\mu_R}$ to be the composition of the product of induction homomorphisms
$$
0\oplus\prod_{n\geq N_R}i^{\Ga_n}:K_*(P_R(G_n))\to \prod_{n\geq N_R} K_*^{\Gamma_n}(\wgn)
$$
(this uses line (\ref{rsplit}) above), the product of equivariant assembly maps 
$$
\prod_{n\geq N_R}\mu_{\Gamma_n}: \prod_{n\geq N_R} K_*^{\Gamma_n}(\wgn)\to \prod_{n\geq N_R}K_*(C^*(\wgn)^{\Gamma_n}),
$$
and the inclusion-and-quotient map
$$
\prod_{n\geq N_R}K_*(C^*(\wgn)^{\Gamma_n})\cong K_*\Big(\prod_{n\geq N_R}C^*(\wgn)^{\Gamma_n}\Big)\to K_*\Big(\frac{\prod_{n} C^*(\wgn)^{\Gamma_n}}{\oplus_{n} C^*(\wgn)^{\Gamma_n}}\Big)
$$
(where the first isomorphism uses stability of the algebras $C^*(\wgn)^{\Gamma_n}$).

It follows from the definition of assembly that the diagram
$$
\xymatrix{ K_*(P_R(X)) \ar@{=}[d] \ar[rr]^{\widetilde{\mu_R}} & & K_*\Big(\frac{\prod_{n} C^*(\wgn)^{\Gamma_n}}{\oplus_{n} C^*(\wgn)^{\Gamma_n}}\Big) \ar@{=}[d] \\
K_*(P_R(X)) \ar[r]^{\mu_R} & K_*(C^*(X)) \ar[r]^{\phi_*~~~} & K_*\Big(\frac{\prod_{n} C^*(\wgn)^{\Gamma_n}}{\oplus_{n} C^*(\wgn)^{\Gamma_n}}\Big) },
$$
commutes.  Now, this diagram implies that if $[p]$ were the image of some element $[x]\in K_0(P_R(X))$ then
$$T(\phi_*[p])=T(\widetilde{\mu_R}[x]);$$
by Theorem \ref{atiyahthe}, however, the right hand side
is the same as $d_*[p]$.  Taking the limit as $R$ tends to infinity yields the lemma.
\end{proof}

\begin{proof}[Proof of Theorem \ref{negthe}]
Using Lemmas \ref{ghostlem} and \ref{atlem}, it suffices to show that for a non-compact ghost projection $p$, $d_*([p])\neq 0$.  This is true for any non-compact projection $p\in C^*(X)$, however.  Indeed, if $P_n\in\mathcal{B}(l^2(X))$ is the projection onto $l^2(G_n)$, then for any $T\in C^*(X)$, $d(T)$ is given by the sequence of cutdowns
$$
[P_0TP_0,~P_1TP_1,~...]\in\frac{\prod_n\mathcal{K}(l^2(G_n,\Hi))}{\oplus_n \mathcal{K}(l^2(G_n,\Hi))}.
$$
However, any $T\in C^*(X)$ satisfies $[T,P_n]\to0$ as $n\to\infty$ (i.e.\ $T$ `asymptotically commutes' with the sequence $(P_n)$).  Hence if $p\in C^*(X)$ is a projection, then $d(p)$ is equal to an element 
$$
[p_0,p_1,...]\in \frac{\prod_n\mathcal{K}(l^2(G_n,\Hi))}{\oplus_n \mathcal{K}(l^2(G_n,\Hi))},
$$ 
where all the $p_i$s are projections: this follows as the asymptotic commutativity property implies that $P_npP_n$ gets arbitrarily close to some actual projection $p_n$ as $n\to\infty$.  Moreover, such a sequence $[p_0,p_1,...]$ will have only finitely many non-zero terms if and only if $p$ is compact.
\end{proof}

We conclude this section with an additional result: it is in some ways weaker than Theorem \ref{negthe} as it only applies to the basic Kazhdan projection, but has the advantage of applying in the non-bounded-geometry case.

\begin{nonbg}\label{nonbg}
Let $X=\sqcup G_n$ be a weak expander as in Definition \ref{exdef} with large girth.  Then the basic Kazhdan projection associated to $X$ is not in the image of the coarse assembly map.
\end{nonbg}

\begin{proof}
The proof is the same as that for Theorem \ref{negthe}, using Lemma \ref{kazdies} in place of Lemma \ref{ghostlem}, and using Lemma \ref{treelem} and the girth property to show that the sequence of universal covers $\widetilde{X}=({\wgn})_{n\in\N}$ is asymptotically faithful and has uniform operator norm localization.
\end{proof}

\section{Injectivity of the coarse assembly map}\label{injsec}

Our main result in this section is the following theorem.

\begin{posthe}\label{posthe}
Say $X=\sqcup G_n$ is a sequence of finite graphs with large girth.  Then the coarse Novikov conjecture holds for $X$.
\end{posthe}

\begin{injgen}\label{injgen}
The methods below, which are similar to those from \cite{Guentner:2008gd}, could be used to get somewhat more general results, for example by combining the results of \cite{Guentner:2008gd} for groups with the possibility of considering families of covering spaces as in the current context.  However, we could not give a particularly `clean' statement along these lines, so restrict ourselves to the theorem above.
\end{injgen}

The proof requires some preliminaries.  Assume throughout this section that $X=\sqcup G_n$ is as in the statement of the theorem.  Let $\wgn$ be the universal cover of $G_n$; note that the girth assumption implies that $\widetilde{X}:=(\wgn)_{n\in\N}$ is an asymptotically faithful covering sequence for $X$ as in Example \ref{cfex} (i).  

In \cite{Guentner:2008gd}, Guentner, Tessera and the second author consider a commutative diagram very close to that below (precisely, in \cite{Guentner:2008gd} a uniform product is used in the bottom left corner; the non-uniform product is more convenient for our current purposes, however, and makes no real difference).
\begin{equation}\label{asymcom}
\xymatrix{ 0 \ar[d]&  \\ 
K_*(P_R(X_{N_R}))\oplus\oplus_{n\geq N_R}K_*(P_R(G_n)) \ar[d] \ar[r]& K_*(C^*(X_{N_R}))\oplus\oplus_{n\geq N_R}K_*(C^*(G_n)) \ar[d] \\ 
K_*(P_R(X)) \ar[d] \ar[r] & K_*(C^*(X)) \ar[d]^{\phi_*} \\
\frac{\prod_n K_*^{\Gamma_n}(P_R(\wgn))}{\oplus_n K_*^{\Gamma_n}(P_R(\wgn))} \ar[r] \ar[d] & K_*\Big(\frac{\prod_n C^*(\wgn)^{\Gamma_n}}{\oplus_n C^*(\wgn)^{\Gamma_n}}\Big) \\ 0 & }
\end{equation}
Here $N_R$ is such that $d(G_n,G_m)\geq R$ for all $m$ and all $n\geq N_R$, and $X_{N_R}=\sqcup_{n=0}^{N_R}G_n\subseteq X$.  The horizontal maps are assembly maps or products of assembly maps as appropriate.  The sequence on the left is exact, as argued in \cite{Guentner:2008gd}, and the top horizontal map and top right vertical map are respectively an isomorphism and an injection as $R\to\infty$, again as argued in that paper (none of these facts are very difficult).  The proof of Theorem \ref{posthe} (which is the statement that the central horizontal map is injective) will thus be completed by the following lemma and a diagram chase.

\begin{asinj}\label{asinj}
For any $R\geq 1$, the assembly map
$$
\mu_\infty:\frac{\prod_n K_*^{\Gamma_n}(P_R(\wgn))}{\oplus_n K_*^{\Gamma_n}(P_R(\wgn))} \to K_*\Big(\frac{\prod_n C^*(\wgn)^{\Gamma_n}}{\oplus_n C^*(\wgn)^{\Gamma_n}}\Big)
$$
(closely related to the \emph{relative assembly map} of \cite[Section 4]{Guentner:2008gd}) is injective.
\end{asinj}

The following proposition is originally due to Pimsner--Voiculescu \cite{Pimsner:1982pt} and Kasparov \cite{Kasparov:1984dw} (both of whom actually proved that the map is an isomorphism).  There are now many relatively elementary proofs due to many different authors, however, especially if one is only interested in injectivity: for example, it follows from the many available proofs of the coarse Baum-Connes conjecture for a tree.

\begin{hkthe}\label{hkthe}
Let $\Ga_n$, $\wgn$ be as in the above.  Then for any $R\geq 1$ the equivariant assembly map
$$
\mu_{\Gamma_n}:K^{\Gamma_n}_*(P_R(\wgn)) \to K_*(C^*(\wgn)^{\Gamma_n})
$$
is injective.
\end{hkthe}

\begin{proof}
Using the fact that $\wgn$ is a tree on which $\Gamma_n$ acts freely and properly, for any $R\geq 1$, $P_R(\wgn)$ is a contractible simplicial complex on which $\Ga_n$ acts freely, properly and compactly.  Hence $P_R(\wgn)$ is a (cocompact) example of a universal proper $\Gamma_n$-space for any $R\geq 1$, whence the map above identifies with the Baum-Connes assembly map for the free group $\Gamma_n$ in this case.  This, however, is injective (in fact, an isomorphism) by the references cited above.
\end{proof}

\begin{proof}[Proof of Lemma \ref{asinj}]
Note first that using the degeneracy of the $K$-theory six term exact sequence associated to the short exact sequence
$$
0\to\oplus_n C^*(\wgn)^{\Gamma_n}\to\prod_n C^*(\wgn)^{\Gamma_n}\to \frac{\prod_n C^*(\wgn)^{\Gamma_n}}{\oplus_n C^*(\wgn)^{\Gamma_n}}\to 0,
$$
and stability of the algebras $C^*(\wgn)^{\Gamma_n}$, there is an isomorphism in $K$-theory
$$
K_*\Big(\frac{\prod_n C^*(\wgn)^{\Gamma_n}}{\oplus_n C^*(\wgn)^{\Gamma_n}}\Big)\cong\frac{\prod_n K_*(C^*(\wgn)^{\Gamma_n})}{\oplus_n K_*(C^*(\wgn)^{\Gamma_n})}.
$$
Using this isomorphism to identify the two groups involved, we see that $\mu_\infty$ identifies with the map
$$
\frac{\prod\mu_{\Ga_n}}{\oplus\mu_{\Ga_n}}:\frac{\prod_n K_*^{\Gamma_n}(P_R(\wgn))}{\oplus_n K_*^{\Gamma_n}(P_R(\wgn))} \to \frac{\prod_n K_*(C^*(\wgn)^{\Gamma_n})}{\oplus_n K_*(C^*(\wgn)^{\Gamma_n})},
$$
where the maps $\mu_{\Gamma_n}$ are the assembly maps from Proposition \ref{hkthe}; as these are all injective, $\mu_\infty$ is also injective.
\end{proof}

\section{The Baum-Connes conjecture with coefficients, and `Gromov monsters'}\label{bcsec}

The \emph{Baum-Connes conjecture with coefficients} predicts that the Baum-Connes assembly map
$$
\mu:\lim_{R\to\infty}KK^\Ga_*(C_0(P_R(\Ga)),A)\to K_*(A\rtimes_r\Ga)
$$
is an isomorphism for any discrete group $\Ga$ and $\Ga$-$C^*$-algebra $A$ (it also makes sense for non-discrete groups -- see \cite[Section 9]{Baum:1994pr}).

In \cite{Higson:2002la}, Higson, Lafforgue and Skandalis used certain groups that Gromov had shown to exist \cite{Gromov:2003gf} (a complete exposition of Gromov's construction of these so-called \emph{Gromov monster} groups is available in \cite{Arzhantseva:2008bv}, due to Arzhantseva and Delzant) to construct counterexamples to the Baum-Connes conjecture with coefficients.  Precisely, they show that for a Gromov monster group $\Ga$ there exist (commutative) $\Ga$-$C^*$-algebras $A_1,A_2$ such that either: the Baum-Connes assembly map with coefficients in $A_1$ fails to be injective; or, the Baum-Connes assembly map with coefficients in $A_2$ fails to be surjective.  

In this section we use our results on expanders with large girth to deduce somewhat more refined results: we show that for a Gromov monster group $\Ga$ there exists a $\Ga$-$C^*$-algebra $A$ such that the Baum-Connes assembly map with coefficients in $A$ is injective and not surjective.  Moreover, applying the main result from the second paper in this series \cite{Willett:2010zh}, we also deduce that the maximal Baum-Connes assembly map with coefficients in $A$ is an isomorphism.  The existence of an example for which the maximal Baum-Connes assembly map is an isomorphism, but the usual version is not, is perhaps surprising.  Note that one can assume that a Gromov monster $\Ga$ has property (T), in which case the maximal Baum-Connes assembly map is certainly not surjective for $\Ga$ itself (yet is injective, using that $\Ga$ is a direct limit of hyperbolic groups).

The following definition makes clear what we mean by `Gromov monster'.  Recall that any countable discrete group $\Gamma$ can be equipped with a left-invariant (i.e.\ $d(gx,gy)=d(x,y)$ for all $g,x,y\in\Gamma$) and proper (i.e.\ balls are finite) metric; moreover, such a metric is unique up to coarse equivalence.  When speaking of metric properties of $\Gamma$ we are implicitly assuming that such a metric has been chosen (which such metric makes no difference to the coarse geometry).  Write $l(g):=d(e,g)$ for the length function associated to $d$.

\begin{gm}\label{gm}
A countable discrete group $\Gamma$ is called a \emph{Gromov monster} if there exists an expander $X=\sqcup G_n$ with large girth and a coarse embedding $f:X\to\Gamma$.
\end{gm}
\noindent
The girth condition (actually something rather stronger -- see \cite{Arzhantseva:2008bv} for details) is necessary to make Gromov's construction work; thus we are not really imposing extra conditions on `Gromov monsters' by assuming it.

\begin{bccoeff}\label{bccoeff}
Say $\Gamma$ is a Gromov monster.  Then there exists a $C^*$-algebra $A$ equipped with a $\Gamma$-action such that:
\begin{enumerate}[(i)]
\item the Baum-Connes assembly map with coefficients in $A$ 
$$
\mu:\lim_{R\to\infty}KK_*^\Gamma(C_0(P_R(\Gamma)),A)\to K_*(A\rtimes_r\Gamma)
$$
fails to be surjective;
\item the Baum-Connes assembly map with coefficients in $A$ is injective;
\item the maximal Baum-Connes assembly map with coefficients in $A$
$$
\mu:\lim_{R\to\infty}KK_*^\Gamma(C_0(P_R(\Gamma)),A)\to K_*(A\rtimes_{max}\Gamma)
$$
is an isomorphism.
\end{enumerate}
\end{bccoeff}
\noindent
Part (iii) is a corollary of our result on the maximal coarse assembly map for expanders with large girth in the second paper of this series \cite{Willett:2010zh}, which we use here without proof.

The choice of the coefficient module $A$ is such that it `captures' the coarse geometric information that we have been studying in the previous sections.  

\begin{adef}\label{adef}
Let $\Gamma$ be a Gromov monster, equipped with a coarse embedding of an expander
$$f:X\to \Gamma$$ 
as in the definition.  For each $n\in\N$, define
$$
X_n=\{g\in\Gamma~|~d(g,f(X))\leq n\},
$$
i.e.\ $X_n$ is the $n$-neighbourhood of $f(X)$ in $\Gamma$.  Let $A_n=l^\infty(X_n,\mathcal{K})\subseteq l^\infty(\Ga,\mathcal{K})$.  Note that $(A_n)_{n\in\N}$ is a directed system; we may thus define.
$$
A=\lim_{n\to\infty}A_n
$$
(equivalently, $A$ is the $C^*$-subalgebra of $l^\infty(\Ga,\mathcal{K})$ generated by the $A_n$s).
\end{adef} 

Note now that if $g,x\in\Ga$ then $d(x,xg)=l(g)$.  Hence the natural right action of $g\in \Ga$ on $l^\infty(\Ga,\mathcal{K})$ maps $A_n$ into $A_{n+l(g)}$; in particular, this action preserves $A$, so restricts to a (right) $\Ga$ action on $A$. 

The following basic properties of this action on $A$, together with Theorems \ref{negthe} and \ref{posthe} and \cite[Theorem 1.1]{Willett:2010zh}, are essentially all we need to complete the proof of Theorem \ref{bccoeff}

\begin{alem}\label{alem}
If $A\rtimes_r\Ga$, $A\rtimes_{max}\Ga$ denote the reduced and maximal crossed products of $A$ with respect to the action described above, then there are canonical isomorphisms
$$
A\rtimes_r\Gamma\cong\lim_{n\to\infty}C^*(X_n)~~\text{ and }~~A\rtimes_{max}\Gamma\cong \lim_{n\to\infty}C^*_{max}(X_n).
$$
\end{alem}

\begin{proof}
Let $A_{alg}$ be the algebraic direct limit of the $A_n$ (equivalently, $*$-algebra generated by all the $A_n$s in $l^\infty(\Ga,\mathcal{K})$).  As any $*$-representation of $A_{alg}$ extends uniquely to a $*$-representation of $A$ and the maximal closure of the algebraic limit $\lim_{n\to\infty}\C[X_n]$ is canonically $*$-isomorphic to the direct limit $\lim_{n\to\infty}C^*_{max}(X_n)$, it suffices to prove that the algebraic crossed product $A_{alg}\rtimes_{alg}\Ga$ is $*$-isomorphic to the algebraic direct limit $\lim_{n\to\infty}\C[X_n]$.  This, however, is clear by restricting the standard identification
$$
\C[|\Ga|]\cong l^\infty(\Ga,\mathcal{K})\rtimes_{alg}\Ga
$$
where the algebraic crossed product is taken with respect to the right $\Ga$ action on itself; recall here that `$\C[|\Ga|]$' denotes the algebraic Roe algebra of $\Ga$, as opposed to $\C[\Ga]$, which denotes the group algebra of $\Ga$, i.e.\ $|\Ga|$ denoted the group thought of as a metric space.
\end{proof}  

\begin{alem2}\label{alem2}
The (maximal) coarse assembly maps 
$$
\mu_n:\lim_{R\to\infty}K_*(P_R(X_n)) \to K_*(C^*_{(max)}(X_n))
$$
form a natural directed system.  The direct limit of this system, say
\begin{equation}\label{infassem}
\mu_\infty:\lim_{n\to\infty}\lim_{R\to\infty}K_*(P_R(X_n))\to \lim_{n\to\infty}K_*(C^*_{(max)}(X_n))
\end{equation}
is naturally isomorphic to the (maximal) Baum-Connes assembly map
\begin{equation}\label{coeffassem}
\mu:\lim_{R\to\infty}KK_*^\Gamma(C_0(P_R(X)),A)\to K_*(A\rtimes_{r~(max)}\Gamma).
\end{equation}
\end{alem2}

\begin{proof}
These identifications follow from a slight elaboration of the arguments in \cite{Yu:1995zl} that identify the (maximal) Baum-Connes assembly map for a group $\Ga$ with coefficients in $l^\infty(X,\mathcal{K})$ with the (maximal) coarse assembly map for the space $|\Ga|$. 
\end{proof}

\begin{proof}[Proof of Theorem \ref{bccoeff}]
The coarse Baum-Connes conjecture is naturally functorial under coarse maps.  As the inclusions $X_n\hookrightarrow X_m$ for $m\geq n$ are all coarse equivalences, it follows that the map 
$$
\mu_\infty:\lim_{n\to\infty}\lim_{R\to\infty}K_*(P_R(X_n))\to \lim_{n\to\infty}K_*(C^*_{(max)}(X_n))
$$
from line (\ref{infassem}) above is equivalent to any of the maps 
$$
\mu_n:\lim_{R\to\infty}K_*(P_R(X_n)) \to K_*(C^*_{(max)}(X_n))
$$
and indeed to the (maximal) coarse assembly map 
$$
\mu:\lim_{R\to\infty}K_*(P_R(X))\to K_*(C_{(max)}^*(X))
$$
for $X$ itself (using that $f$ is a coarse equivalence).  Parts (i), (ii) and (iii) of Theorem \ref{bccoeff} are now immediate from Theorems \ref{negthe}, \ref{posthe} and \cite[Theorem 1.1]{Willett:2010zh}, respectively.
\end{proof}

\appendix

\section{The uniform case}\label{unicase}

In this appendix, we collect together the adjustments, and additional ingredients, that are necessary to extend Theorems \ref{negthe} and \ref{posthe} and \cite[Theorem 1.1]{Willett:2010zh} to the case of the uniform coarse assembly map.  

\begin{unialg}\label{unialg}
Let $X$ be a proper metric space, and fix $Z$ a countable dense subset of $X$.  Let $T=(T_{x,y})_{x,y\in Z}$ be the matrix representation of a bounded operator on $l^2(Z,\Hi)$ with respect to the natural basis of $l^2(Z)$, so each $T_{x,y}$ is an element of $\mathcal{B}(\Hi)$.  $T$ is said to be a \emph{uniform operator} if for all $\epsilon>0$ there exists $N\in\N$ such that for all $x,y$ there exists $F_{x,y}\in\mathcal{B}(l^2(Z,\Hi))$ of rank at most $N$ such that $\|T_{x,y}-F_{x,y}\|<\epsilon$, and if for any bounded subset $B\subseteq X$, the set
$$
\{(x,y)\in (B\times B)\cap (Z\times Z)~|~T_{x,y}\neq 0\}
$$
is finite.  The \emph{propagation of $T$} is 
$$
\text{prop}(T):=\inf\{S>0~|~T_{x,y}=0 \text{ for all } x,y\in Z \text{ with } d(x,y)>S\}.
$$
The \emph{algebraic uniform algebra of $X$}, denoted $U\C[X]$, is the $*$-subalgebra of $\mathcal{B}(l^2(Z,\Hi))$ consisting of locally compact, finite propagation operators.   The \emph{uniform algebra of $X$}, denoted $UC^*(X)$, is the closure of $U\C[X]$ inside $\mathcal{B}(l^2(Z,\Hi))$, and the \emph{maximal uniform algebra of $X$}, denoted $UC^*_{max}(X)$, is the completion of $U\C[X]$ for the obvious universal norm (this exists in the bounded geometry case).

Assume now in addition that $X$ is uniformly discrete.  Define $\C_u[X]$ to be the $*$-subalgebra of $l^2(X)$ consisting of finite propagation operators (with the obvious analogue of the definition above).  The \emph{uniform Roe algebra of $X$}, denoted $C^*_u(X)$, is the norm closure of $\C_u[X]$ in $\mathcal{B}(l^2(X))$ and the \emph{maximal uniform Roe algebra}, denoted $C^*_{u,max}(X)$, is the completion of $\C[X]$ for the obvious maximal norm (this again exists in the bounded geometry case).
\end{unialg}

\begin{unimor}[\cite{Willett:2010ca}, Section 4.2]\label{unimor}
There exist canonical Morita equivalences 
$$
UC^*_{(max)}(X)\stackrel{M}{\sim}C^*_{u,(max)}(X)
$$
in both the maximal and non-maximal cases. \qed
\end{unimor}

Recall that \v{S}pakula \cite{Spakula:2009tg} has defined the \emph{uniform $K$-homology} groups of a topological space $X$, denoted $K_*^u(X)$, and uniform coarse assembly maps
$$
\mu:\lim_{R\to\infty}K_*^u(P_R(X))\to K_*(UC^*(X));
$$
one may similarly define a maximal uniform coarse assembly map
$$
\mu:\lim_{R\to\infty}K_*^u(P_R(X))\to K_*(UC^*_{max}(X)).
$$
Using Proposition \ref{unimor} above and the main result of \cite{Exel:1993pt}, the image of either of these assembly maps can equally be taken in $K_*(C^*_{u,(max)}(X))$, which is in fact closer to \v{S}pakula's treatment in \cite{Spakula:2009tg}).  One can now insist on uniformly finite rank for all of the various versions of the Roe algebras in this piece, and follow through all of the arguments above (using the stability of the algebra $UC^*_{(max)}$ and the Morita equivalence above to get around the fact that $C^*_{u,max}(X)$ is not stable).  One place where one must be somewhat careful is in the definition of the localization algebra (cf.\ \cite[Section 4]{Willett:2010zh}): here one starts with functions $f:[0,\infty)\to U\C[X]$ with propagation tending to zero, but \emph{does not} demand that the rank of approximants remain bounded across all $f(t)$; this is necessary for the Eilenberg swindle type arguments from \cite[Section 5]{Willett:2010zh} to go through.  Another place where some care is required is showing that the Dirac-dual-Dirac argument carries through in this context; this was carried out by \v{S}pakula and the first author in \cite[Section 4]{Willett:2010ca}.

As a result, we have the following theorem; it admits various generalizations and modifications as mentioned in the sections above and in \cite[Remark 3.1]{Willett:2010zh}, but we focus here on the main statements.  The third part uses the techniques of the second paper in this series \cite{Willett:2010zh}.

\begin{unitheo}\label{unitheo}
Let $X=\sqcup G_n$ be a bounded geometry space of graphs with large girth. Then:
\begin{enumerate}[(i)]
\item if $X$ is an expander, then the uniform coarse assembly map for $X$ is not surjective;
\item the uniform coarse assembly map for $X$ is injective;
\item the maximal uniform coarse assembly map for $X$ is an isomorphism. \qed
\end{enumerate}
\end{unitheo}

The following corollary says that our results on the Baum-Connes conjecture for Gromov monsters can be made to hold with \emph{commutative} coefficients.  See \cite[Section 10]{Spakula:2009tg} for the connection between the uniform coarse assembly map for the metric space $|\Ga|$ underlying a group $\Ga$ and the Baum-Connes assembly map for $\Ga$ with coefficients in $l^\infty(\Ga)$.

\begin{unicor}\label{unicor}
Let $\Ga$ be a Gromov monster group as in Definition \ref{gm} above, and let $A_u$ be as in Definition \ref{adef}, but with the scalars $\C$ replacing the compact operators $\mathcal{K}$ (so in particular, $A_u$ is a commutative $C^*$-algebra).  Then the Baum-Connes assembly map for $\Ga$ with coefficients in $A_u$ is injective but not surjective, and the maximal Baum-Connes assembly map for $\Ga$ with coefficients in $A_u$ is an isomorphism.   \qed
\end{unicor}

\bibliography{Generalbib}

\begin{thebibliography}{10}

\bibitem{Arzhantseva:2008bv}
G.~Arzhantseva and T.~Delzant.
\newblock Examples of random groups.
\newblock Available on the authors' websites, 2008.

\bibitem{Atiyah:1976th}
M.~Atiyah.
\newblock Elliptic operators, discrete groups and von {N}eumann algebras.
\newblock {\em Asterisque}, 32-33:43--72, 1976.

\bibitem{Baum:1988sf}
P.~Baum and A.~Connes.
\newblock ${K}$-theory for discrete groups.
\newblock In {\em Operator algebras and applications}, volume~1 of {\em London
  Math. Soc. Lecture Note Ser. 135}, pages 1--20. Cambridge University Press,
  1988.

\bibitem{Baum:2000ez}
P.~Baum and A.~Connes.
\newblock Geometric ${K}$-theory for {L}ie groups and foliations.
\newblock {\em Enseign. Math. (2)}, 46:3--42, 2000 (first circulated 1982).

\bibitem{Baum:1994pr}
P.~Baum, A.~Connes, and N.~Higson.
\newblock Classifying space for proper actions and ${K}$-theory of group
  ${C}^*$-algebras.
\newblock {\em Contemporary Mathematics}, 167:241--291, 1994.

\bibitem{Baum:1980pt}
P.~Baum and R.~G. Douglas.
\newblock ${K}$-homology and index theory.
\newblock In {\em Operator algebras and applications, Part I}, volume~38 of
  {\em Proc. Sympos. Pure Math.}, pages 117--173. American Mathematical
  Society, 1980.

\bibitem{Baum:2007ek}
P.~Baum, N.~Higson, and T.~Schick.
\newblock On the equivalence of geometric and analytic ${K}$-homology.
\newblock {\em Pure Appl. Math. Q.}, 3(1):1--24, 2007.

\bibitem{Baum:2009hq}
P.~Baum, N.~Higson, and T.~Schick.
\newblock A geometric description of equivariant ${K}$-homology for proper
  actions.
\newblock Available on the second and third authors' websites, 2009.

\bibitem{Bekka:2000kx}
B.~Bekka, P.~de~la Harpe, and A.~Valette.
\newblock {\em Kazhdan's Property (T)}.
\newblock Cambridge University Press, 2008.

\bibitem{Chen:2008so}
X.~Chen, R.~Tessera, X.~Wang, and G.~Yu.
\newblock Metric sparsification and operator norm localization.
\newblock {\em Adv. Math.}, 218(5):1496--1511, 2008.

\bibitem{Chen:2004bd}
X.~Chen and Q.~Wang.
\newblock Ideal structure of uniform {R}oe algebras of coarse spaces.
\newblock {\em J. Funct. Anal.}, 216(1):191--211, 2004.

\bibitem{Chen:2005cy}
X.~Chen and Q.~Wang.
\newblock Ghost ideals in uniform {R}oe algebras of coarse spaces.
\newblock {\em Arch. Math. (Basel)}, 84(6):519--526, 2005.

\bibitem{Connes:1994zh}
A.~Connes.
\newblock {\em Noncommutative Geometry}.
\newblock Academic Press, 1994.

\bibitem{Exel:1993pt}
R.~Exel.
\newblock A {F}redholm operator approach to {M}orita equivalence.
\newblock {\em ${K}$-theory}, 7(3):285--308, May 1993.

\bibitem{Gong:2008ja}
G.~Gong, Q.~Wang, and G.~Yu.
\newblock Geometrization of the strong {N}ovikov conjecture for residually
  finite groups.
\newblock {\em J. Reine Angew. Math.}, 621:159--189, 2008.

\bibitem{Gromov:2003gf}
M.~Gromov.
\newblock Random walks in random groups.
\newblock {\em Geom. Funct. Anal.}, 13(1):73--146, 2003.

\bibitem{Guentner:2008gd}
E.~Guentner, R.~Tessera, and G.~Yu.
\newblock Operator norm localization for linear groups and its applications to
  ${K}$-theory.
\newblock {\em Adv. Math.}, doi:10.1016/j.aim.2010.10.022, 2010.

\bibitem{Higson:1999km}
N.~Higson.
\newblock Counterexamples to the coarse {B}aum-{C}onnes conjecture.
\newblock Preprint, 1999.

\bibitem{Higson:2001eb}
N.~Higson and G.~Kasparov.
\newblock ${E}$-theory and ${KK}$-theory for groups which act properly and
  isometrically on {H}ilbert space.
\newblock {\em Invent. Math.}, 144:23--74, 2001.

\bibitem{Higson:1999be}
N.~Higson, G.~Kasparov, and J.~Trout.
\newblock A {B}ott periodicity theorem for infinite dimensional {H}ilbert
  space.
\newblock {\em Adv. Math.}, 135:1--40, 1999.

\bibitem{Higson:2002la}
N.~Higson, V.~Lafforgue, and G.~Skandalis.
\newblock Counterexamples to the {B}aum-{C}onnes conjecture.
\newblock {\em Geom. Funct. Anal.}, 12:330--354, 2002.

\bibitem{Higson:1995fv}
N.~Higson and J.~Roe.
\newblock On the coarse {B}aum-{C}onnes conjecture.
\newblock {\em London Mathematical Society Lecture Notes}, 227:227--254, 1995.

\bibitem{Higson:2000bs}
N.~Higson and J.~Roe.
\newblock {\em Analytic ${K}$-homology}.
\newblock Oxford University Press, 2000.

\bibitem{Higson:1993th}
N.~Higson, J.~Roe, and G.~Yu.
\newblock A coarse {M}ayer-{V}ietoris principle.
\newblock {\em Math. Proc. Cambridge Philos. Soc.}, 114:85--97, 1993.

\bibitem{Kasparov:1975ht}
G.~Kasparov.
\newblock Topological invariants of elliptic operators {I}: ${K}$-homology.
\newblock {\em Math. USSR-Izv.}, 9(4):751--792, 1975.

\bibitem{Kasparov:1984dw}
G.~Kasparov.
\newblock Lorentz groups: ${K}$-theory of unitary representations and crossed
  products ({E}nglish translation).
\newblock {\em Sov. Math. Dokl.}, 29:256--260, 1984.

\bibitem{Kasparov:1988dw}
G.~Kasparov.
\newblock Equivariant ${KK}$-theory and the {N}ovikov conjecture.
\newblock {\em Invent. Math.}, 91(1):147--201, 1988.

\bibitem{Lubotzky:1994tw}
A.~Lubotzky.
\newblock {\em Discrete Groups, Expanding Graphs and Invariant Measures}.
\newblock Birkh\"{a}user, 1994.

\bibitem{Lubotzky:1988fu}
A.~Lubotzky, R.~Phillips, and P.~Sarnak.
\newblock {R}amanujan graphs.
\newblock {\em Combinatorica}, 8(3):261--277, 1988.

\bibitem{Margulis:1973lh}
G.~Margulis.
\newblock Explicit constructions of expanders ({R}ussian).
\newblock {\em Problemy Pereda\v{c}i Informacii}, 9(4):71--80, 1973.

\bibitem{Milnor:1971kl}
J.~Milnor.
\newblock {\em Introduction to Algebraic ${K}$-theory}.
\newblock Annals of Mathematics Studies. Princeton University Press, 1971.

\bibitem{Oyono-Oyono:2009ua}
H.~Oyono-Oyono and G.~Yu.
\newblock {K}-theory for the maximal {R}oe algebra of certain expanders.
\newblock {\em J. Funct. Anal.}, 257(10):3239--3292, 2009.

\bibitem{Pimsner:1982pt}
M.~Pimsner and D.-V. Voiculescu.
\newblock ${K}$-groups of reduced crossed products by free groups.
\newblock {\em J. Operator Theory}, 8:131--156, 1982.

\bibitem{Roe:1989sh}
J.~Roe.
\newblock Partitioning non-compact manifolds and the dual {T}oeplitz problem.
\newblock In D.~Evans and M.~Takesaki, editors, {\em Operator Algebras and
  Applications}. Cambridge University Press, 1989.

\bibitem{Roe:1993lq}
J.~Roe.
\newblock Coarse cohomology and index theory on complete {R}iemannian
  manifolds.
\newblock {\em Mem. Amer. Math. Soc.}, 104(497), July 1993.

\bibitem{Roe:1996dn}
J.~Roe.
\newblock {\em Index Theory, Coarse Geometry and Topology of Manifolds},
  volume~90 of {\em CBMS Conference Proceedings}.
\newblock American Mathematical Society, 1996.

\bibitem{Roe:2002nx}
J.~Roe.
\newblock Comparing analytic assembly maps.
\newblock {\em Quart. J. Math. Oxford Ser. (2)}, 53(2):1--8, 2002.

\bibitem{Roe:2003rw}
J.~Roe.
\newblock {\em Lectures on Coarse Geometry}, volume~31 of {\em University
  Lecture Series}.
\newblock American Mathematical Society, 2003.

\bibitem{Spakula:2009rr}
J.~\v{S}pakula.
\newblock Non-${K}$-exact uniform {R}oe ${C^*}$-algebras.
\newblock To appear in Journal of ${K}$-theory; available on the author's
  website, 2009.

\bibitem{Spakula:2009tg}
J.~\v{S}pakula.
\newblock Uniform ${K}$-homology theory.
\newblock {\em J. Funct. Anal.}, 257(1):88--121, 2009.

\bibitem{Willett:2010ca}
J.~\v{S}pakula and R.~Willett.
\newblock Maximal and reduced {R}oe algebras of coarsely embeddable spaces.
\newblock 26 pages; submitted; available on the authors' websites, 2010.

\bibitem{Wang:2007uf}
Q.~Wang.
\newblock Remarks on ghost projections and ideals in the {R}oe algebras of
  expander sequences.
\newblock {\em Arch. Math. (Basel)}, 89(5):459--465, 2007.

\bibitem{Willett:2010zh}
R.~Willett.
\newblock Higher index theory for certain expanders and {G}romov monster groups
  {II}.
\newblock available on the first author's website, 2010.

\bibitem{Yu:1995zl}
G.~Yu.
\newblock {B}aum-{C}onnes conjecture and coarse geometry.
\newblock {\em ${K}$-theory}, 9(3):223--231, 1995.

\bibitem{Yu:1995bv}
G.~Yu.
\newblock Coarse {B}aum-{C}onnes conjecture.
\newblock {\em ${K}$-theory}, 9:199--221, 1995.

\bibitem{Yu:1997kb}
G.~Yu.
\newblock Localization algebras and the coarse {B}aum-{C}onnes conjecture.
\newblock {\em ${K}$-theory}, 11(4):307--318, 1997.

\bibitem{Yu:1998wj}
G.~Yu.
\newblock The {N}ovikov conjecture for groups with finite asymptotic dimension.
\newblock {\em Ann. of Math.}, 147(2):325--355, 1998.

\bibitem{Yu:200ve}
G.~Yu.
\newblock The coarse {B}aum-{C}onnes conjecture for spaces which admit a
  uniform embedding into {H}ilbert space.
\newblock {\em Invent. Math.}, 139(1):201--240, 2000.

\end{thebibliography}

\end{document}